\documentclass[12pt]{article}
\usepackage{fullpage}
\usepackage[french,english]{babel}
\usepackage[utf8]{inputenc}
\usepackage[T1]{fontenc}
\usepackage{amsthm}
\usepackage{amssymb}
\usepackage{amsmath}
\usepackage{stmaryrd}
\usepackage{makeidx}
\usepackage{graphicx}
\usepackage{caption}

\theoremstyle{plain}
\usepackage[colorlinks=true,breaklinks=true,linkcolor=black]{hyperref} 
\usepackage{amsfonts}
\newtheorem{thm}{Theorem}[section] 
\newtheorem{prop}[thm]{Proposition}
\newtheorem{cor}[thm]{Corollary}

\newtheorem{lemme}[thm]{Lemma}

\date{}
\makeindex
\usepackage{mathrsfs}
\usepackage[all]{xy}
\usepackage{dsfont}
\theoremstyle{definition}
\newtheorem{defn}[thm]{Definition}
\newtheorem{defn/prop}[thm]{Definition/Proposition}
\newtheorem{lem/def}[thm]{Lemma/Definition}
\newtheorem{defn/thm}[thm]{Definition/Theorem}

\newtheorem{exmp}[thm]{Example}
\newtheorem{rque}[thm]{Remark}
\usepackage{geometry}
\title{Completed Iwahori-Hecke algebras and parahoric Hecke algebras for Kac-Moody groups over local fields}
\author{ Ramla \textsc{Abdellatif}\\ LAMFA -- UPJV  \\UMR CNRS 7352, 80 039 AMIENS Cedex 1, France\\ ramla.abdellatif@u-picardie.fr\\  \and Auguste \textsc{H{\'e}bert} \\Univ Lyon, UJM-Saint-Etienne CNRS\\ UMR CNRS 5208, F-42023, SAINT-ETIENNE, France\\ auguste.hebert@ens-lyon.fr}
\makeatletter \@addtoreset{figure}{section}\makeatother

\newcommand{\R}{\mathbb{R}}
\newcommand{\A}{\mathbb{A}}
\newcommand{\N}{\mathbb{N}}
\newcommand{\Z}{\mathbb{Z}}

\newcommand{\C}{\mathbb{C}}

\newcommand{\I}{\mathcal{I}}
\newcommand{\T}{\mathcal{T}}

\newcommand{\q}{\mathfrak{q}}

\newcommand{\F}{\mathcal{F}}

\newcommand{\conv}{\mathrm{conv}}

\newcommand{\supp}{\mathrm{supp}}

\newcommand{\HH}{\mathcal{H}}
\newcommand{\RR}{\mathcal{R}}

\newcommand{\wt}{\mathbf{w}}
\newcommand{\vt}{\mathbf{v}}
\newcommand{\ut}{\mathbf{u}}
\newcommand{\AAA}{\mathcal{A}}
\newcommand{\BLHH}{{^{\mathrm{BL}}\mathcal{H}}}
\newcommand{\FHH}{{^\mathrm{F}}\mathcal{H}}
\newcommand{\opp}{\mathrm{opp}}
\begin{document}

\maketitle

\begin{abstract}
Let $G$ be a split Kac-Moody group over a non-Archimedean local field. We define a completion of the Iwahori-Hecke algebra of $G$, then we compute its center and prove that it is isomorphic (via the Satake isomorphism) to the spherical Hecke algebra of $G$. This is thus similar to the situation for reductive groups. Our main tool is the masure $\mathcal{I}$ associated to this setting, which plays here the same role as Bruhat-Tits buildings do for reductive groups. In a second part, we associate a Hecke algebra to each spherical face $F$ of type $0$, extending a construction that was only known, in the Kac-Moody setting, for the spherical subgroup and for the Iwahori subgroup.
\end{abstract}

\section{Introduction}
Let $\mathbf{G}_0$ be a split reductive group over a non-Archimedean local field $\mathcal{K}$ and set $G_0=\mathbf{G}_0(\mathcal{K})$. An important tool to study complex representations of $G_{0}$ are Hecke algebras attached to each open compact subgroup of $G_{0}$: if $K$ is such a subgroup, the Hecke algebra $\mathcal{H}_K$ associated to $K$ is the convolution algebra of complex-valued $K$-bi-invariant functions on $G_0$ with compact support. Two choices of $K$ are of particular interest: the first one is when $K=K_s=\mathbf{G}(\mathcal{O})$ with $\mathcal{O}$ being the ring of integers of $\mathcal{K}$. In this case, $K_{s}$ is a maximal open compact subgroup of $G$ and $\mathcal{H}_s=\mathcal{H}_{K_s}$ is a commutative algebra called the \textbf{spherical Hecke algebra} of $G_0$. This algebra can be explicitly described through the Satake isomorphism: indeed, if $W^{v}$ denotes the Weyl group of $G_{0}$ and $Q^{\vee}$ is the coweight lattice of $G_{0}$, then $\mathcal{H}_s$ is isomorphic to the subalgebra $\C[Q^\vee]^{W^v}$ of $W^{v}-$invariant elements in the group algebra of $(Q^\vee,+)$. A second interesting choice is when $K = K_{I}$ is an Iwahori subgroup of $G_0$: then $\mathcal{H}:=\HH_{K_I}$ is called the \textbf{Iwahori-Hecke algebra} of $G_0$. This algebra comes with a basis (called the Bernstein-Lusztig basis) indexed by the affine Weyl group of $G_{0}$ and such that the product of two elements of this basis can be expressed via the Bernstein-Lusztig presentation \cite{lusztig1989affine}. This presentation enables us to compute the center of~$\mathcal{H}$ and to check that it is isomorphic to the spherical Hecke algebra of $G_{0}$. These results can be summarized as follows:
\[\HH_s \xrightarrow[\sim]{S}\C[Q^\vee]^{W^v}\overset{g}{\rightarrow} \HH, \quad\text{and}\quad\mathrm{Im}(g)=\mathcal{Z}(\HH),
\]
where $S$ denotes the Satake isomorphism and $g$ comes from the Bernstein-Lusztig basis.

This article aims to study how far this theory can be extended to split Kac-Moody groups. Among the different definitions of Kac-Moody groups that are available in the literature, we choose to use the definition given by Tits in \cite{tits1987uniqueness} as it is more algebraic. Given $\mathbf{G}$ a split Kac-Moody group over $\mathcal{K}$, set $G=\mathbf{G}(\mathcal{K})$. To study~$G$, Gaussent and Rousseau built in \cite{gaussent2008kac} an object $\I=\I(G)$ that they called a \textbf{masure} (also known as \textbf{affine ordered hovel}), and later extended by Rousseau \cite{rousseau2016groupes, rousseau2017almost}. This masure is a generalization of Bruhat-Tits buildings introduced in \cite{bruhat1972groupes, bruhat1984groupes} as it gives back the Bruhat-Tits building of $G$ when $\mathbf{G}$ is reductive. As a set, $\I$ is a union of apartments that are all isomorphic to a standard one (denoted by $\A$ in the sequel) and $G$ acts on $\I$. We still have an arrangement of hyperplanes, called \textbf{walls}, but in general this arrangement is not locally finite anymore. This explains why faces in $\A$ are not sets anymore but filters. Another main difference with Bruhat-Tits buildings is that in general, two points of $\I$ are not necessarily contained in a common apartment.

Analogues of $K_{s}$ and $K_{I}$ (and more generally of parahoric subgroups) can be defined as fixers of some specific faces in $\I$. When $\mathbf{G}$ is an affine Kac-Moody group, Braverman, Kazhdan and Patnaik attached to $G$ a spherical Hecke algebra and an Iwahori-Hecke algebra \cite{braverman2011spherical, braverman2016iwahori}, and they obtained a Satake isomorphism as well as Bernstein-Lusztig relations. All these results were generalized to arbitrary Kac-Moody groups by Bardy-Panse, Gaussent and Rousseau \cite{gaussent2014spherical, bardy2016iwahori}. In this framework, the Satake isomorphism appears as an isomorphism between $\HH_{s}:=\mathcal{H}_{K_s}$ and $\C\llbracket Y\rrbracket^{W^v}$, where $Y$ denotes a lattice that can be, in first approximation, thought of as the coroot lattice (though it can be really different, even in the affine case) and~$\C\llbracket Y\rrbracket$ is its Looijenga's algebra (which is a completion of the group algebra~$\C[Y]$, see Definition~\ref{DefLooijenga algebra}). So far, the analogy with the reductive case stops here. Indeed, let $\HH$ be the Iwahori-Hecke algebra of $G$: following what happens in the reductive case \cite{parkinson2006buildings}, one would expect the center of $\HH$ to be isomorphic to the spherical Hecke algebra $\HH_{s}$. Unfortunately, this is not the case, as we prove that the center of~$\HH$ is actually more or less trivial (see Lemma~\ref{lemCentre IH usuel}). Moreover, in the non-reductive case, $\C\llbracket Y\rrbracket^{W^v}$ is a set of infinite formal series that cannot embed into $\HH$, where all elements have finite support. All these reasons explain why we define a completion~$\hat{\mathcal{H}}$ of $\mathcal{H}$ as follows: letting $(Z^\lambda H_w)_{\lambda\in Y^+, w\in W^v}$ be the Bernstein-Lusztig basis of $\HH$, we define $\hat{\mathcal{H}}$ as the set of formal series $\hat{\mathcal{H}}$ whose support satisfies some conditions similar to what appears in the definition of $\C\llbracket Y\rrbracket$. One of our main results states that $\hat{\mathcal{H}}$ can be turned into an algebra when endowed with a well-defined convolution product compatible with the canonical inclusion of $\HH$ into $\hat{\mathcal{H}}$ (see Theorem~\ref{Thm Convolution pour les familles sommables} and Corollary~\ref{CorCompleted IH est une algebre}). We then determine the center of $\hat{\mathcal{H}}$ and show that it is isomorphic to $\C\llbracket Y\rrbracket^{W^v}$, as wanted (see Theorem~\ref{ThmCentreIHcompletee}). As before, these results can be summarized as follows:
\[\HH_s \xrightarrow[\textstyle\sim]{\textstyle~S~}\C\llbracket Y\rrbracket^{W^v} \overset{g}{\rightarrow}{\hat{\HH}}, \quad\text{and}\quad\mathrm{Im}(g)=\mathcal{Z}({\hat{\HH}}),\]
where $S$ denotes the Satake isomorphism and $g$ comes again from the Bernstein-Luzstig basis.

Another part of this paper is devoted to the construction of Hecke algebras attached to more general subgroups than $K_{s}$ and $K_{I}$. Recall that $K_{s}$ is the fixer of $\{0\}$ while~$K_{I}$ is the fixer of the type $0$ chamber $C^{+}_{0}$. When $G$ is reductive, any face~$F$ between~$\{0\}$ and $C_{0}^{+}$ corresponds to an open compact subgroup of $G$ (namely the parahoric subgroup associated with $F$) contained in its fixer $K_{F}$, hence one can use it to attach a Hecke algebra to $F$. This explains why it seems natural, in the non-reductive case, to wonder whether one can attach a Hecke algebra to $K_{F}$ for all faces~$F$ between~$\{0\}$ and $C_{0}^{+}$. We succeed in defining such an algebra when $F$ is spherical, which means that its fixer under the action of the Weyl group is finite. Our construction is very close to what is done for the Iwahori-Hecke algebra in \cite{bardy2016iwahori}. We also prove that when $F$ is not spherical and different from $\{0\}$ (that cannot happen if $G$ is affine), this construction fails because the structural constants are infinite.

Finally, recall that up to now, there was no known topology on $G$ that generalizes the usual topology on $\mathbf{G}_0(\mathcal{K})$, for $\mathbf{G}_0$ a split reductive group over $\mathcal{K}$, in which~$K_{s}$ and~$K_{I}$ are open compact subgroups of $\mathbf{G}_0(\mathcal{K})$. We prove (see Theorem~\ref{Thm Pas de topologie ouverte compacte pour les fixateurs}) that when $\mathbf{G}$ is not reductive, there is no way to turn $G$ into a topological group such that~$K_{s}$ or $K_{I}$ (or, more generally, any given parahoric subgroup of $G$) is open \hbox{compact}. This result implies in particular that one cannot define smooth representations of $G$ in the same way as in the reductive case.

The paper is organized as follows: we first recall the definition of masures in Section~\ref{secApartment standard}. The reader only interested in Iwahori-Hecke algebras can read the two first sections and the last one, and skip the rest of this section. Section~\ref{secRestriction topologique} is devoted to prove that $G$ cannot be turned into a topological group in which $K_{s}$ or $K_{I}$ is open compact. In Section~\ref{secAlgebre d'IH completee}, we define the completed Iwahori-Hecke algebra $\hat{\mathcal{H}}$ of $G$ and compute its center, as well as the center of $\HH$. Finally, we use Section~\ref{secAlgebres de Hecke generales} to attach a Hecke algebra to any spherical face between $\{0\}$ and $C_{0}^{+}$ and to prove that this construction fails if $F$ is not spherical and different from $\{0\}$.

\begin{rque}
As explained in Section~\ref{secApartment standard}, this paper is actually written in a more general framework, as we only need $\I$ to be an abstract masure and $G$ to be a strongly transitive group of (positive, type-preserving) automorphisms of $\I$. In particular, this applies to almost split (and not only split) Kac-Moody groups over local fields.
\end{rque}

\subsubsection*{Acknowledgements} We warmly thank St\'ephane Gaussent for suggesting our collaboration, for multiple discussions and for his useful comments on previous versions of this manuscript. We also thank Nicole Bardy-Panse and Guy Rousseau for discussions on this topic and for their comments on a previous version of this paper. Finally, we thank the referee for their valuable comments and suggestions, and for his/her interesting questions.

\section{Masures: general framework}
\label{secApartment standard}
We recollect here some well-known facts. Further details are available in the first two sections of \cite{rousseau2011masures}.
\subsection{Root generating system and Weyl groups}
\label{subRootGenSyst}
A \textbf{Kac-Moody matrix} (or \textbf{generalized Cartan matrix}) is a square matrix $A=(a_{i,j})_{i,j\in I}$ indexed by a finite set $I$, with integral coefficients, and such that:
\begin{enumerate}
\item
$\forall i\in I,\ a_{i,i}=2$;

\item
$\forall (i,j)\in I^2, (i \neq j) \Rightarrow (a_{i,j}\leq 0)$;

\item
$\forall (i,j)\in I^2,\ (a_{i,j}=0) \Leftrightarrow (a_{j,i}=0$).
\end{enumerate}
A \textbf{root generating system} is a $5$-tuple $\mathcal{S}=(A,X,Y,(\alpha_i)_{i\in I},(\alpha_i^\vee)_{i\in I})$
 made of a Kac-Moody matrix $A$ indexed by the finite set $I$, of two dual free $\Z$-modules $X$ and~$Y$ of finite rank, and of a free family $(\alpha_i)_{i\in I}$ (respectively $(\alpha_i^\vee)_{i\in I}$) of elements in~$X$ (resp. $Y$) called \textbf{simple roots} (resp. \textbf{simple coroots}) that satisfy $a_{i,j}=\alpha_j(\alpha_i^\vee)$ for all~$i,j$ in $I$. Elements of $X$ (respectively of $Y$) are called \textbf{characters} (resp. \textbf{cocharacters}).

Fix such a root generating system $\mathcal{S}=(A,X,Y,(\alpha_i)_{i\in I},(\alpha_i^\vee)_{i\in I})$ and set $\A:=Y\otimes \R$. Each element of $X$ induces a linear form on $\A$, hence $X$ can be seen as a subset of the dual $\A^*$. In particular, the $\alpha_{i}$'s (with $i \in I$) will be seen as linear forms on $\A$. This allows us to define, for any $i \in I$, an involution $r_{i}$ of $\A$ by setting $r_{i}(v):= v-\alpha_i(v)\alpha_i^\vee$ for any $v \in \A$. Note that the points fixed by $r_{i}$ are exactly the elements of $\ker \alpha_{i}$. We define the \textbf{Weyl group of $\mathcal{S}$} as the subgroup $W^{v}$ of $\mathrm{GL}(\A)$ generated by the finite set $\{r_{i}, \ i \in I\}$. The pair $(W^{v}, \{r_{i}, \ i \in I\})$ is a Coxeter system, hence we can consider the length $\ell(w)$ with respect to $\{r_{i}, \ i \in I\}$ of any element $w$ of $W^{v}$.

For any $x \in \A$, we set $\underline{\alpha}(x)=(\alpha_i(x))_{i\in I}\in \R^I$. Let $P^{\vee}:= \{v\in \A \mid \underline{\alpha}(v)\in \Z^I\}$ be the \textbf{coweight-lattice}, which is not a lattice when $ \A_\mathrm{in}:= \bigcap_{i\in I}\ker \alpha_{i}$ is non-zero, $Q^{\vee}:= \bigoplus_{i\in I}\Z\alpha_{i}^{\vee}$ be the \textbf{coroot-lattice} and $Q^{\vee}_{\mathbb{R}}= \bigoplus_{i\in I}\R\alpha_{i}^{\vee}$. Furthermore setting $Q_{+}^{\vee}:= \bigoplus_{i\in I}\N \alpha_{i}^{\vee}$, we can define a pre-order $\leq_{Q^\vee}$ on $\A$ as follows: for any $x,y \in \A$, we say that $x\leq_{Q^{\vee}}y$ iff $y-x\in Q^{\vee}_{+}$. We also set $Q^{\vee}_{-}:=-Q^{\vee}_{+}$ for future reference.

There is an action of the Weyl group $W^{v}$ on $\A^{*}$ given by the following formula:
\[ \forall x \in \A,\, w \in W^{v},\, \alpha \in \A^{*}, \quad (w\cdot \alpha)(x):= \alpha(w^{-1}\cdot x).\]
Let $\Phi:= \{w\cdot \alpha_i\mid (w,i)\in W^{v}\times I\}$ be the set or \textbf{real roots}: then $\Phi$ is a subset of $Q:= \bigoplus_{i\in I}\Z\alpha_i$. We will also use the set $\Delta:= \Phi \cup \Delta_\mathrm{im}^{+} \cup \Delta_\mathrm{im}^{-} \subset Q$ of all roots, as defined in \cite{kac1994infinite}. Note that $\Delta$ is stable under the action of $W^{v}$. For any root $\alpha \in \Delta$, we set $\Lambda_{\alpha}:= \Z$ if $\alpha \in \Phi$, and $\Lambda_{\alpha}:= \R$ otherwise (\textit{i.e} if $\alpha \in \Delta_\mathrm{im}:= \Delta_\mathrm{im}^{+} \cup \Delta_\mathrm{im}^{-}$). For any pair $(\alpha, k) \in \Delta \times \Z$, we set
\begin{equation}
\label{definition D(alpha k)}
D(\alpha,k):= \{t\in \A \mid \alpha(t)+k\geq 0 \} \quad\text{and}\quad D^\circ(\alpha,k):=\{t\in \A \mid \alpha(t)+k > 0\}.
\end{equation}
For any root $\alpha \in \Delta$, we also set $D(\alpha,+\infty):= \A$.

Finally, we let $W:=Q^{\vee}\rtimes W^{v}\subset\mathrm{GA}(\A)$ be the \textbf{affine Weyl group of $\mathcal{S}$}, where $\mathrm{GA}(\A)$ denotes the group of affine isomorphisms of $\A$. Note that $W \subset P^{\vee} \rtimes W^{v}$ and that $\alpha(P^{\vee})$ is contained in $\Z$ for any $\alpha \in Q$. Consequently, if $\tau$ is a translation of $\A$ of vector $p \in P^{\vee}$, then for any $\alpha \in Q$, $\tau$ acts by permutations on the set $\{ D(\alpha,k), \ k \in \Z\}$. On the other hand, as $W^{v}$ stabilizes $\Delta$, any element of $W^{v}$ permutes the sets of the form $D(\alpha,k)$, where $\alpha$ runs over $\Delta$. Hence we have an action of $W$ on $\{D(\alpha, k), \ (\alpha, k) \in \Delta \times \Z \}$.

\subsection{Vectorial faces and Tits cone}
\label{subsecVectorialFaces}
As in the reductive case, define the \textbf{fundamental chamber} as $C_{f}^{v}:= \{v\in \A \mid \forall i \in I,\, \alpha_i(v)>0\}$. For any subset $J$ in $I$, set
\[
F^{v}(J)=\{v\in \A \mid \forall j \in J, \, \alpha_{j}(v)=0 \ \text{ and } \forall i \in I \setminus J,\, \alpha_{i}(v)>0\};
\]
then the closure $\overline{C_{f}^{v}}$ of $C_{f}^{v}$ is exactly the union of all $F^{v}(J)$'s for $J \subset I$. The \textbf{positive vectorial faces} (resp. \textbf{negative vectorial faces}) are defined as the sets $w\cdot F^{v}(J)$ (resp. $-w\cdot F^{v}(J)$) for $w \in W^{v}$ and $J \subset I$, and a \textbf{vectorial face} is either a positive vectorial face or a negative one. A \textbf{positive chamber} (resp. \textbf{negative chamber}) is a cone of the form $w\cdot C^{v}_{f}$ (resp. $-w\cdot C^{v}_{f}$) for some $w \in W^{v}$. Note that for any $x \in C^{v}_{f}$ and any $w \in W^{v}$, we have $(w\cdot x = 1) \Rightarrow (w = 1)$, which ensures that the action of $w \in W^{v}$ on the set of positive chambers is simply transitive.

Let $\mathcal{T}:= \bigcup_{w\in W^{v}} w\cdot \overline{C^{v}_{f}}$ be the \textbf{Tits cone} and $-\mathcal{T}$ be the negative cone. We can use it to define a $W^{v}$-invariant pre-order $\leq$ on $\A$ as follows: $$\forall (x,y) \in \A^{2}, \quad x \leq y \iff y-x \in \mathcal{T}.$$ We also set $Y^{+}:= \T \cap Y$ and $Y^{++}:= Y \cap \overline{C^{v}_{f}}$. We can now recall the following simple but very useful result \cite[Lem.\,2.4 a)]{gaussent2014spherical}.
\begin{lemme}
\label{lemLemme 2.4 a) de GR14}
For any $\lambda\in Y^{++}$ and any $w\in W^v$, we have $w\cdot \lambda \leq_{Q^\vee} \lambda$.
\end{lemme}

\subsection{Filters and masures}
This section aims to recall what masures are. As stated in the introduction, the reader only interested in the completion of Iwahori-Hecke algebras can skip this section and go directly to Section~\ref{secAlgebre d'IH completee}.

Masures were first introduced for symmetrizable split Kac-Moody groups over a valued field whose residue field contains $\C$ by Gaussent and Rousseau \cite{gaussent2008kac}. Later, Rousseau axiomatized this construction in \cite{rousseau2011masures}, then generalized it with further developments to almost-split Kac-Moody groups over non-Archimedean local fields in \cite{rousseau2016groupes, rousseau2017almost}. For the reader familiar with this work, let us mention that we consider here semi-discrete masures which are thick of finite thickness.

\subsubsection{Filters, sectors and rays}
\begin{defn}
A \textbf{filter on a set $E$} is a non-empty set $F$ of non-empty subsets of~$E$ that satisfies the following conditions:
\begin{itemize}
\item for any subsets $S_{1}, S_{2}$ of $E$ that both belong to $F$, then $S_{1} \cap S_{2}$ belongs to $F$;
\item for any subsets $S_{1}, S_{2}$ of $E$ with $S_{1}$ in $F$ and $S_{1} \subset S_{2}$, then $S_{2}$ belongs to $F$.
\end{itemize}
\end{defn}

Given a filter $F$ on a set $E$ and a subset $E'$ of $E$, we say that \textbf{$F$ contains $E'$} if every element of $F$ contains $E'$. If $E'$ is non-empty, then the set $F_{E'}$ of all subsets of~$E$ containing $E'$ is a filter on $E$ called the \textbf{filter associated to $E'$}. By language abuse and to ease notations, we will sometimes write that $E'$ is a filter, by identification of~$E'$ with $F_{E'}$.

Now, let $F$ be a filter on a finite-dimensional real affine space $E$. We define its \textbf{closure} $\overline{F}$ as the filter of all subsets of $E$ that contain the closure of some arbitrary element of $F$, and its \textbf{convex hull} $\mathrm{conv}(F)$ as the filter of all subsets of $E$ that contain the convex hull of some arbitrary element of $F$. Said differently, we have:
\[ \overline{F}:= \ \{ S \subset E \mid \exists S' \in F, \ \overline{S'} \subset S \} \text{ and} \ \mathrm{conv}(F):= \{S \subset E \mid \exists S' \in F, \ \mathrm{conv}(S') \subset S \}.
\]

Given two filters $F_{1}$ and $F_{2}$ on the same set $E$, we say that \textbf{$F_{1}$ is contained in $F_{2}$} iff any subset of $E$ contained $F_{2}$ is in $F_{1}$. Similarly, we say that the filter $F_{1}$ is \textbf{contained in a subset} $\Omega$ of $E$ iff any subset of $E$ contained in $\Omega$ is in $F_{1}$.

Let $\Omega$ be a subset of $\A$ containing an element $x$ in its closure $\overline{\Omega}$ of $\Omega$. The \textbf{germ of $\Omega$ in $x$} is defined as the filter $\mathrm{germ}_{x}(\Omega)$ of all subsets of $\A$ containing some neighborhood of $x$ in $\Omega$. A \textbf{sector} in $\A$ is a translate $\mathfrak{s}:= x + C^{v}$ of a vectorial chamber $C^{v}= \pm w\cdot C_{f}^{v}$ (with $w \in W^{v}$) by an element $x \in \A$. The point $x$ is called the \textbf{base point} of the sector~$\mathfrak{s}$ and the chamber $C^{v}$ is called its \textbf{direction}. One can easily check that the intersection of two sectors having the same direction is a sector with the same direction.

Given a sector $\mathfrak{s}:= x + C^{v}$ as above, the \textbf{sector-germ of $\mathfrak{s}$} is the filter $\mathfrak{S}$ of all subsets of $\A$ containing an $\A$-translate of~$\mathfrak{s}$. Note that it only depends on the direction $C^{v}$ of~$\mathfrak{s}$. In particular, we denote by $\pm \infty$ the sector-germ of $\pm C^{v}_{f}$.

Finally, let $\delta$ be a ray with base point $x$ and let $y \not= x$ be another point on $\delta$ (which amounts to say that $\delta$ contains the interval $]x,y]=[x,y]\setminus\{x\}$, as well as the interval $[x,y]$). We say that $\delta$ is \textbf{pre-ordered} (resp. \textbf{generic}) if either $y \leq x$ or $x \leq y$ (resp. if $y-x\in \pm\mathring \T$, where $\pm\mathring \T$ denotes the interior of the cone $\pm \T$).

\subsubsection{Enclosures and faces}
We keep the notations introduced at the end of Section~\ref{subRootGenSyst}. Given a filter $F$ on $\A$, we define its \textbf{enclosure} $\mathrm{cl}_{\A}(F)$ as the filter of all subsets of $\A$ containing some element of $F$ of the form $ \bigcap_{\alpha \in \Delta} D(\alpha, k_{\alpha})$, with $k_{\alpha} \in \Z \cup \{+\infty\}$ for any $\alpha \in \Delta$. Sets of the form $D(\alpha,k)$ with $\alpha \in \Phi$ and $k \in \Z$ are called \textbf{half-apartments} in $\A$. Sets of the form $M(\alpha, k):= \{ t \in \A \mid \alpha(t) + k = 0\}$ with $\alpha \in \Phi$ and $k \in \Z$ are called \textbf{walls} in $\A$.

A \textbf{local face in $\A$} is a filter $F^{\ell}$ of the form $\mathrm{germ}_{x}(x + F^{v})$, where $x \in \A$ is called the \textbf{vertex} of $F^{\ell}$ and a vectorial face $F^{v} \subset \A$ is a vectorial face called the \textbf{direction} of $F^{\ell}$. To keep track of the elements $x$ and $F^{v}$, such a local face may be denoted $F^{\ell}(x,F^{v})$. A local face is said to be \textbf{spherical} when its direction is spherical: in this case, its pointwise stabilizer under the action of $W^{v}$ is a finite group.

A \textbf{face in $\A$} is a filter $F = F(x, F^{v})$ associated with a point $x \in A$ and a vectorial face $F^{v} \subset A$ as follows: a subset $S$ of $\A$ belongs to $F$ iff it contains an intersection of half-spaces $D(\alpha,k_{\alpha})$ or open half-spaces $\mathring D(\alpha, k_{\alpha})$, with $\alpha \in \Delta$ and $k_{\alpha} \in \Z$, that also contains the local face $F^{\ell}(x,F^{v})$. Note that (local) faces can be ordered as follows: given two such faces $F,F'$ in $\A$, we say that \textbf{$F$ is a face of $F'$} (or \textbf{$F'$ contains $F$}, or~\textbf{$F'$~dominates $F$}) when $F \subset \overline{F'}$.

As explained at the end of Section~\ref{subRootGenSyst}, the action of $W$ on $\A$ permutes the sets of the form $D(\alpha, k)$, where $(\alpha, k)$ runs over $\Delta \times \Z$. In particular, this implies that $W$ permutes enclosures (resp. walls, faces) of $\A$.

The \textbf{dimension} of a face $F$ is the smallest dimension of an affine space generated by some element $S$ of $F$. Such an affine space is unique and is called the \textbf{support} of $F$. A \textbf{local chamber} (or local alcove) is a maximal local face, \textit{i.e} a local face of the form $F^{\ell}(x,\pm w\cdot C^{v}_{f})$ for $x\in \A$ and $w\in W^{v}$. The \textit{fundamental local chamber} is $C_{0}^{+}:=F^{\ell}(0,C^{v}_{f})$. A \textbf{local panel} is a spherical local face which is maximal among faces that are not chambers. Equivalently, a local panel is a spherical local face of dimension $\dim \A - 1$. Analogue definitions of chambers and panels exist, see for instance \cite[\S 1.4]{gaussent2008kac}. Finally, a local face is \textbf{of type $0$} (or: is a \textbf{type $0$ local face}) if its vertex lies in $Y$. We denote by $F_{0}$ the local face $F^{\ell}(0, \A_\mathrm{in})$, where $\A_\mathrm{in} = F^{v}(I) = \bigcap_{i \in I} \ker(\alpha_{i})$. From now on, we will write \textbf{type $0$ face} instead of \textbf{type $0$ local face} to make it shorter.

\begin{rque}
In \cite{rousseau2011masures}, Rousseau defines a notion of \textbf{chimney} that he uses in his axiomatization of masures. We do not define here what chimneys are: we only recall that each sector-germ is a splayed, solid chimney-germ, that each spherical face is contained in a solid chimney and that the action of $W$ on $\A$ permutes the chimneys and preserve their properties (being splayed or solid for instance). For more details about this, see \cite[\S1.10]{rousseau2011masures}.
\end{rque}

\subsubsection{Apartments and masure of type $\A$}
An \textbf{apartment of type $\A$} is a set $A$ together with a non-empty set $\mathrm{Isom}(\A, A)$ of bijections (called \textbf{Weyl-isomorphisms}) such that, given $f_{0} \in \mathrm{Isom}(\A,A)$, the elements of $\mathrm{Isom}(\A,A)$ are exactly the bijections of the form $f_{0} \circ w$ with $w \in W$. All the isomorphisms considered in the sequel will be Weyl-isomorphisms, hence we will only write isomorphism instead of Weyl-isomorphism.

An isomorphism $\phi: A \to A' $ between two apartments of type $\A$ is a bijection such that: $f \in \mathrm{Isom}(\A, A) \iff \phi \circ f \in \mathrm{Isom}(\A, A')$. By construction, all the notions that are preserved under the action of $W$ can be extended to any apartment of type $A$. For instance, we can define sectors, enclosures, faces or chimneys in any apartment $A$ of type $\A$, as well as a pre-order $\leq_{A}$ on $A$.

We can now define the most important object of this section: the masures of type~$\A$.
\begin{defn}
A \textbf{masure of type $\A$} is a set $\mathcal{I}$ endowed with a covering $\mathcal{A}$ of subsets (called \textbf{apartments}) such that the five following axioms hold.

(MA1) Any $A\in \mathcal{A}$ can actually be endowed with a structure of apartment of type~$\A$.

(MA2) If $F$ is a point (resp. a germ of a preordered interval, a generic ray, a solid chimney) contained in an apartment $A$ and if $A'$ is an other apartment containing $F$, then $A\cap A'$ contains $\mathrm{cl}_A(F)$ and there exists an isomorphism from $A$ to $A'$ that fixes $\mathrm{cl}_A(F)$.

(MA3) If $\mathfrak{R}$ is the germ of a splayed chimney and if $F$ is a face or a germ of a solid chimney, then there exists an apartment that contains both $\mathfrak{R}$ and $F$.

(MA4) If two apartments $A$, $A'$ contain both $\mathfrak{R}$ and $F$ as in (MA3), then there exists an isomorphism from $A$ to $A'$ that fixes $\mathrm{cl}_A(\mathfrak{R}\cup F)$.

(MAO) If $x$ and $y$ are two points that are both contained in two apartments $A$ and~$A'$ and such that $x\leq_{A} y$, then the two segments $[x,y]_A$ and $[x,y]_{A'}$ are equal.
\end{defn}
Recall that saying that an apartment contains a germ of a filter means that it contains at least one element of this germ. Similarly, a map fixes a germ when it fixes at least one element of this germ.

From now on, $\I$ will denote a masure of type $\A$. We assume that $\I$ is thick of \textbf{finite thickness}, which means that the number of chambers ($={}$alcoves) containing a given panel is finite and greater or equal to three. We also assume that there exists a group $G$ of automorphisms of $\I$ that acts strongly transitively on $\I$, which implies that all the isomorphisms involved in the axioms above are induced by the action of elements of $G$. We fix an apartment in $\I$ that we identify with $\A$ and call the \textbf{fundamental apartment of $\I$}. As the action of $G$ on $\I$ is strongly transitive, the apartments of $\I$ are exactly the sets $g\cdot\A$ with $g \in G$. Let $N$ be the stabilizer of $\A$ in $G$: it defines a group of affine automorphisms of $\A$, and we assume that this group is $W^{v} \ltimes Y$. As we will see in Section~\ref{subsecMasureAssociee}, these assumptions are not very restrictive for our purpose as they are all satisfied by the masure attached to a split Kac-Moody group $G$ over a non-Archimedean local field (see also \cite{gaussent2008kac} and \cite{rousseau2016groupes}).

\begin{rque}
\label{RemSimplificationAxiomes}
In a recent work \cite[\S5]{hebert2017convexity}, the second author gives a much simpler axiomatic for masures. To simplify the arguments, the reader can assume that~$G$ is an affine Kac-Moody group, in which case the three axioms (MA2), (MA4) and (MAO) can be replaced by the following statement \cite[Th.\,5.38]{hebert2017convexity}: \textbf{for any two apartments~$A$ and~$A'$ in $\I$, we have $A \cap A' = \mathrm{cl}(A \cap A')$ and there exists an isomorphism from~$A$ to~$A'$ that fixes $A \cap A'$.} This partially explains why the affine case is less technical, as it does not require to question the existence of isomorphisms that fix subsets of an intersection of apartments.
\end{rque}

\begin{rque}
\label{RemFaceLocaleFixee}
Let $F$ be a local face of an apartment $A$ and $A'$ be another apartment that contains $F$. Then $F$ is also a local face of $A'$ and there exists an isomorphism from $A$ to $A'$ that fixes $F$. Indeed, if $x$ is the vertex of $F$ and if $J$ is a germ of a pre-ordered segment based at $x$ and contained in $F$, then the enclosure of $J$ contains~$F$ and the application of (MA2) to $J$ now proves the claim.
\end{rque}

\begin{rque}
\label{Fixateurs d'une face locale et de sa face}
Pick $w \in W$ and $\mathcal{F}$ a filter on $\A$ fixed by $w$: then $w$ fixes $\mathrm{cl}(\mathcal{F})$. Combined with the argument used in Remark \ref{RemFaceLocaleFixee}, this proves here that for any vectorial face $F^{v}$ and any base point $x$, the fixer in $G$ of the face $F(x,F^{v})$ is exactly the fixer in $G$ of the corresponding local face $F^{\ell}(x, F^{v})$.
\end{rque}

\begin{rque}
\label{Pre-ordre sur la masure}
As we noticed earlier, each apartment $A$ of $\I$ can be endowed with a pre-order $\leq_{A}$ induced by $\leq_{\A}$. Let $A$ be an apartment of $\I$ and $x,y$ be two points in $A$ such that $x \leq_{A} y$. By \cite[Prop.\,5.4]{rousseau2011masures}, we know that for any apartment $A'$ of~$\I$ that contains both $x$ and $y$, we also have $x \leq_{A'} y$. We hence get a relation $\leq$ on~$\I$ and \cite[Th.\,5.9]{rousseau2011masures} ensures that this relation is a $G$-invariant pre-order on $\I$.
\end{rque}

\subsection{Masure attached to a split Kac-Moody group}
\label{subsecMasureAssociee}
As in \cite{tits1987uniqueness} or in \cite[Chap.\,8]{remy2002groupes}, we consider the group functor $\mathbf{G}$ associated with the root generating system~$\mathcal{S}$ fixed in Section~\ref{subRootGenSyst}. This functor goes from the category of rings to the category of groups and satisfies axioms (KMG1)--(KMG9) of \cite{tits1987uniqueness}. For any field $R$, the group $\mathbf{G}(R)$ is uniquely determined by these axioms \cite[Th.\,1']{tits1987uniqueness}. Furthermore, this functor $\mathbf{G}$ contains a toric functor $\mathbf{T}$ (denoted by $\mathcal{T}$ in \cite{remy2002groupes}) that goes from the category of rings to the category of abelian groups, and two functors~$\mathbf{U}^{\pm}$ going from the category of rings to the category of groups.

In particular, let $\mathcal{K}$ be a non-Archimedean local field. Denote by $\mathcal{O}$ its ring of integers, by $\varpi$ a fixed uniformizer of $\mathcal{O}$, by $q$ the cardinality of the residue class field $\mathcal{O}/\varpi\mathcal{O}$ and set $G:=\mathbf{G}(\mathcal{K})$ (as well as $U^{\pm}:= \mathbf{U}^{\pm}(\mathcal{K})$, $T:= \mathbf{T}(\mathcal{K})$, etc.). For any sign $\varepsilon \in \{+, -\}$ and any root $\alpha \in \Phi^{\varepsilon}$, there is an isomorphism $x_{\alpha}$ from $\mathcal{K}$ to a root group $U_{\alpha}$. For any integer $k \in \Z$, we get a subgroup $U_{\alpha, k}:= x_{\alpha}(\varpi^{k}\mathcal{O})$ of $U_{\alpha}$ (see \cite[\S 3.1]{gaussent2008kac} for precise definitions). Let $\I$ denote the masure attached to $G$ in \cite{rousseau2017almost}: then the following properties hold.
\begin{itemize}
\item The fixer of $\A$ in $G$ is $H:= \mathbf{T}(\mathcal{O})$ \cite[Rem.\,3.2]{gaussent2008kac}.
\item The fixer of $\{0\}$ in $G$ is $K_{s}:=\mathbf{G}(\mathcal{O})$ \cite[Exam.\,3.14]{gaussent2008kac}. Applying (MA2) to~$\{0\}$ and using Remark \ref{Fixateurs d'une face locale et de sa face}, we get that $K_{s}$ is also the fixer in $G$ of the face $F_{0}$.
\item For any pair $(\alpha, k) \in \Phi \times \Z$, the fixer of $D(\alpha, k)$ in $G$ is $H.U_{\alpha,k}$ \cite[Exer.\,4.2.7)]{gaussent2008kac}.
\item For any sign $\varepsilon \in \{+, -\}$, $U^{\varepsilon}$ is the fixer in $G$ of $\varepsilon \infty$ \cite[Exam.\,4.2.4)]{gaussent2008kac}.
\end{itemize}
Moreover, each panel is contained in $q+1$ chambers, hence $\I$ is thick of finite thickness.
\begin{rque}
\label{RemCasReductif}
The group $G$ is reductive iff $W^{v}$ is finite. In this case, $\I$ is the usual Bruhat-Tits building of $G$ and we have $\T = \A$ and $Y^{+} = Y$.
\end{rque}

\paragraph{Acknowledgements} We warmly thank St\'ephane Gaussent for suggesting our col-
laboration, for multiple discussions and for his useful comments on previous versions
of this manuscript. We also thank Nicole Bardy-Panse and Guy Rousseau for dis-
cussions on this topic and for their comments on a previous version of this paper.
Finally, we thank the referee for their valuable comments and suggestions, and for
his/her interesting questions.

\paragraph{Funding} The first author was supported by the ANR grants ANR-14-CE25-0002-01 and ANR-16-CE40-0010-
01, by the GDR TLAG and by a CNRS grant PEPS-JCJC. The second author was supported by
the ANR grant ANR-15-CE40-0012.

\tableofcontents

\section{A topological restriction on parahoric subgroups}
\label{secRestriction topologique}
\subsection{Statement of the result and idea of the proof}
\label{subsecEnonceThmRestriction}
In this section, we will prove that beside the reductive case, it is impossible to endow $G$ with a structure of topological group for which $K_{s}$ or $K_{I}$ are open compact subgroups, where $K_{I}$ denotes the (standard) Iwahori subgroup.\footnote{We recall that $K_{I}$ is the fixer in $G$ of the fundamental local chamber $C_{0}^{+}$.} In fact, we will prove the following result, which is slightly more general.
\begin{thm}
\label{Thm Pas de topologie ouverte compacte pour les fixateurs}
Let $F$ be a type $0$ face of $\A$ and let $K_{F}$ be its fixer in $G$. If $W^{v}$ is infinite, then there is no topology of topological group on $G$ for which $K_{F}$ is open and compact.
\end{thm}
Let $F$ be a type $0$ face of $\A$, \textit{i.e} a local face whose vertex lies in $Y$. First note that, up to replacing $F$ by $h\cdot F$ for some well-chosen $h \in G$, which leads to consider the conjugate of $K_{F}$ under $h$ instead of $K_{F}$, we can assume that $F$ is contained in~$C_{0}^{\pm}$. As the treatment of both cases is similar, we assume that $F$ is contained in~$C_{0}^{+}$. To prove Theorem~\ref{Thm Pas de topologie ouverte compacte pour les fixateurs}, it is enough to prove the existence of $g \in G$ such that $K_{F}/(K_F\cap g\cdot K_F.g^{-1})$ is infinite.

To explain the strategy of proof, we need to introduce some more notations. Let $\alpha \in \Phi^{+}$ and $(w,i) \in W^{v} \times I$ be such that $\alpha = w\cdot \alpha_{i}$. For any $k \in \Z$, set
\[
M^{\alpha}_{k}:=\{t\in \A \mid \alpha(t)= k\},\quad D^{\alpha}_{k}:= \{t\in \A \mid \alpha(t) \leq k\},
\]
and let $K_{\alpha, k}$ be the fixer of $D^{\alpha}_{k}$ in $G$. Furthermore, pick a panel $P^{\alpha}_{k}$ in $M^{\alpha}_{k}$ and a chamber $C^{\alpha}_{k}$ contained in $\mathrm{conv}(M^{\alpha}_{k}, M^{\alpha}_{k+1})$ that dominates $P^{\alpha}_{k}$.

For any $i \in I$, we let $q_{i} + 1$ (resp. $q'_{i} + 1$) be the number of chambers containing~$P^{\alpha_{i}}_{0}$ (resp. $P^{\alpha_{i}}_{1}$). By \cite[Prop.\,2.9]{rousseau2011masures} and \cite[Lem.\,3.2]{hebert2016distances}, $q_{i}$ and $q'_{i}$ do not depend on the choices of the panels $P^{\alpha_{i}}_{0}$ and $P^{\alpha_{i}}_{1}$. (This fact will be explained in the proof of Lemma~\ref{lemOrbite d'un point sous Kj}.) As $\alpha_i(\alpha_i^\vee)=2$, and as there exists an element of $G$ that induces on~$\A$ a~translation of vector $\alpha^{\vee}_{i}$ (because we assumed that the stabilizer $N$ of $\A$ in $G$ induces $W^{v} \ltimes Y$ for group of affine automorphisms), the value of $1 + q_{i}$ (resp. $1 + q'_{i}$) is also the number of chambers that contain $P_{2k}^{\alpha_{i}}$ (resp. $P_{2k+1}^{\alpha_{i}}$) for any integer $k$.

Let us now explain the basic idea of the proof. Pick $g \in G$ such that $g\cdot0 \in C^{v}_{f}$ and set $F':= g\cdot F$: then $K_F/(K_F\cap K_{F'})$ is in bijection with $K_F\cdot F'$. For $\alpha = w\cdot \alpha_{i}$, set $\tilde{K_{\alpha}}:= \bigcup_{k \in \Z} K_{\alpha,k}$: then $\mathbb{T}_\alpha:=\tilde{K_\alpha}\cdot\A$ is a semi-homogeneous extended tree with parameters $q_{i}$ and $q'_{i}$. Using the thickness of $\I$, we can prove that the number $n_{\alpha}$ of walls between $0$ and $g\cdot0$ that are parallel to $\alpha^{-1}(\{0\})$ satisfies $\vert K_{\alpha, 1}g\cdot0 \vert \geq 2^{n_{\alpha}}$, which implies that $\vert K_{F}g\cdot0 \vert \geq 2^{n_{\alpha}}$. As $n_{\alpha}$ can be made arbitrarily large (for~a~suitable choice of $\alpha$) when $W^{v}$ is infinite, this will end the proof.

\subsection{Detailed proof of Theorem \ref{Thm Pas de topologie ouverte compacte pour les fixateurs}}
Fix for now $\alpha = w\cdot \alpha_{i}$. Set $K_{\alpha}:= K_{\alpha, 1}$ and pick a sector-germ $\q$ contained in $D^{\alpha}_{0}$. By (MA3), we know that for any $x \in \I$, there exists an apartment $A_{x}$ that contains both $x$ and $\q$. Axiom (MA4) implies the existence of an isomorphism $\phi_{x}: A_{x} \to \A$ that fixes $\q$, and \cite[\S 2.6]{rousseau2011masures} ensures that $\phi_{x}(x)$ does not depend on the choice of the apartment $A_{x}$ nor on the isomorphism $\phi_{x}$, hence we can denote this element by $\rho_{\q}(x)$. The map $\rho_{q}: \I \to \A$ is the retraction of~$\I$ onto $\A$ centered at $\q$, and its restriction to $\mathbb{T}_{\alpha}$ does not depend on the choice of $\q \in D^{\alpha}_{0}$.

\begin{rque}
\label{Rem restriction retractions aux appartements}
Let $A$ be an apartment containing $\q$ and $\phi = \rho_{q}\vert_{A}$ be the restriction to $A$ of the retraction map $\rho_{q}$. Then $\phi$ is the unique isomorphism of apartments that fixes $A \cap \A$. Indeed, (MA4) implies the existence of an isomorphism of apartments $\psi: A \to \A$ that fixes $\q$. By definition, $\rho_{\q}$ coincides with $\psi$ on $A$, hence $\phi = \psi$ is an isomorphism of apartments, and fixes $\A$, hence $\phi$ fixes $\A \cap A$. If $f: A \to \A$ is an isomorphism of apartments that fixes $A \cap \A$, then $f \circ \phi^{-1}: \A \to \A$ is an isomorphism of affine spaces that fixes $\q$, hence it must be trivial, which proves that $f = \phi$ is unique.
\end{rque}

\begin{lemme}\label{lemKalpha fixe les enclos}
Let $v\in \tilde{K}_\alpha$ and $x\in \A$ be such that $v\cdot x \in \A$. Then $v$ fixes $D^{\alpha}_{\lceil \alpha(x)\rceil }$.
\end{lemme}
\begin{proof}
Set $A:= v\cdot\A$ and let $k \in \Z$ be such that $A \cap \A$ contains $D^{\alpha}_{k}$. By \cite[Lem.\,3.2]{hebert2016distances}, $A \cap \A$ is a half-apartment,\footnote{Our definition of half-apartments is a bit different from the definition of \cite{hebert2016distances}: what we call half-apartments correspond to the \textbf{true} half-apartments of \cite{hebert2016distances}.} hence there exists an integer $k_{1} \in \Z$ such that $A \cap A = D^{\alpha}_{k_{1}}$. Let $\phi: \A \to A$ be the isomorphism of apartments induced by~$v$: Remark~\ref{Rem restriction retractions aux appartements} ensures that $\phi$ fixes $A \cap \A$, which means that $v$ fixes $A \cap \A$. As~$v\cdot x$ belongs to $A \cap \A$, we obtain that $v\cdot (v\cdot x) = v\cdot x$, hence $v\cdot x = x$ belongs to~$D^{\alpha}_{k_{1}}$. This implies that $D^{\alpha}_{\lceil \alpha(x) \rceil}$ is contained in $D^{\alpha}_{k_{1}} = A \cap \A$, hence is fixed by $v$.
\end{proof}

\begin{lemme}\label{lemOrbite d'un point sous Kj}
Let $x \in \A$ and $M_{x}:= M^{\alpha}_{\lceil \alpha(x) \rceil}$. Then the map $f: K_{\alpha}\cdot x \to K_{\alpha}. M_{x}$ defined by $f(v\cdot x):= v\cdot M_{x}$ is well-defined and bijective.
\end{lemme}
\begin{proof}
Let $v, v' \in K_{\alpha}$ be such that $v\cdot x = v'\cdot x$: by Lemma~\ref{lemKalpha fixe les enclos}, we get that $v^{-1}\cdot v'$ fixes $M_{x}$, hence $f$ is well-defined. Now assume that $v,v' \in K_{\alpha}$ satisfy $v\cdot M_{x} = v'.M_{x}$. Set $u:= v^{-1}v'$ and $A:= u\cdot\A$: then $u\cdot M_{x}= M_{x}$ is contained in $A \cap \A$ hence $u$ fixes $D^{\alpha}_{\lceil \alpha(x) \rceil}$ by Lemma \ref{lemKalpha fixe les enclos}. In particular, we have $u\cdot x = x$, \textit{i.e} $v\cdot x = v'\cdot x$ and $f$ is injective. As $f$ is surjective by definition, the lemma is proved.
\end{proof}

Set $\alpha_{\I}:= \alpha \circ \rho_{\q}$ and, for any integer $k \geq 0$, let $\mathcal{C}^{\alpha}_{k}$ be the set of all chambers~$C$ that dominate some element of $K_{\alpha}.P_{k}$ and satisfy $\alpha_{\I}(C) > k$ (which means that there exists $X \in C$ such that $\alpha_{\I}(x) > k$ for all $x \in X$). Assume also that the chamber $C^{\alpha}_{k}$ chosen in Section \ref{subsecEnonceThmRestriction} is not contained in $D^{\alpha}_{k}$.

\begin{lemme}
\label{lemBijection entre les murs et les chambres}
For any integer $k \geq 0$, the map $g_{k}: K_{\alpha}M^{\alpha}_{k+1} \to \mathcal{C}^{\alpha}_{k}$ sending $v\cdot M^{\alpha}_{k+1}$ onto $v\cdot C^{\alpha}_{k}$ is well-defined and bijective.
\end{lemme}
\begin{proof}
The proof of the first part of the assertion is as in Lemma \ref{lemOrbite d'un point sous Kj}: if $v,v' \in K_{\alpha}$ are such that $v\cdot M^{\alpha}_{k+1} = v'.M^{\alpha}_{k+1}$, then $u:=v^{-1}\cdot v'$ satisfies $u\cdot M^{\alpha}_{k+1} \subset \A$ and Lemma~\ref{lemKalpha fixe les enclos} implies that $u\cdot C^{\alpha}_{k} = C^{\alpha}_{k}$, \textit{i.e} $v\cdot C^{\alpha}_{k} = v'\cdot C^{\alpha}_{k}$. As we moreover have $\alpha_{\I}(v\cdot C^{\alpha}_{k}) = \alpha(C^{\alpha}_{k}) > k$, $v\cdot C^{\alpha}_{k} = v'\cdot C^{\alpha}_{k}$ belongs to $\mathcal{C}^{\alpha}_{k}$ and the map $g_{k}$ is well-defined.

Assume now that $v,v' \in K_{\alpha}$ are such that $v\cdot C^{\alpha}_{k} = v'\cdot C^{\alpha}_{k}$. Set $u:= v^{-1}v'$ and let $X \in C^{\alpha}_{k}$ be an element fixed by $u$. Let $x \in X$ be such that $\alpha(x) > k$: by Lemma~\ref{lemKalpha fixe les enclos}, $u$ fixes $M^{\alpha}_{k+1} \subset D^{\alpha}_{\lceil \alpha(x) \rceil}$, hence $g_{k}$ is injective.

It remains to check that $g_{k}$ is surjective, \textit{i.e} that $\mathcal{C}^{\alpha}_{k} = K_{\alpha}\cdot C^{\alpha}_{k}$. If $C \in \mathcal{C}^{\alpha}_{k}$, then there exists $u \in K_{\alpha}$ such that $C$ dominates $u\cdot P^{\alpha}_{k}$. By \cite[Prop.\,2.9.1)]{rousseau2011masures}, there is an apartment $A$ that contains both $u\cdot D^{\alpha}_{k}$ and $C$, and Remark~\ref{Rem restriction retractions aux appartements} now gives an explicit isomorphism $\phi: A \to \A$ that fixes $A \cap \A$. If $v \in K_{\alpha}$ induces $\phi$, then $\alpha_{\I}(C) = \alpha(v\cdot C)$, hence $\alpha(v\cdot C) > k$, and $v\cdot C \subset \A$ dominates $P^{\alpha}_{k}$, hence $v\cdot C = C^{\alpha}_{k}$, \textit{i.e} $C = v^{-1}\cdot C^{\alpha}_{k}$, which ends the proof.
\end{proof}
Combining Lemmas~\ref{lemOrbite d'un point sous Kj} and \ref{lemBijection entre les murs et les chambres}, we get the following corollary.
\begin{cor}
\label{corOrbite d'un point sous l'action des Ualpha}
For any $x\in \A$, if $k:=\max (1,\lceil \alpha(x) \rceil)$, then $|K_\alpha\cdot x|=q_i' q_i q_i' \cdots$ ($k-1$ factors).
\end{cor}

Until the end of this section, we assume that $W^{v}$ is infinite.
\begin{lemme}
\label{lemNonCommensurabiliteParahoriques}
Let $F$ be a type $0$ face of $\A$. If $W^{v}$ is infinite, then there exists $g \in G$ such that $K_{F}/K_{F} \cap K_{g\cdot F}$ is infinite.
\end{lemme}
\begin{proof}
Let $g \in G$ be such that $a:= g\cdot0$ belongs to $C^{v}_{f}$ and set $F':= g\cdot F$. Let $(\alpha_{n})_{n \geq 0}$ be an injective sequence of positive real roots (\textit{i.e} $\alpha_{n} \in \Phi^{+}$ for any non-negative integer $n$). As we have $K_{\alpha_{n}} \subset K_{F}$, hence $|K_F\cdot F'|\geq |K_{\alpha_{n}}.a|$, for all $n \geq 0$, it is enough to check (by Corollary~\ref{corOrbite d'un point sous l'action des Ualpha} and thickness of $\I$) that $\alpha_{n}(a) \to +\infty$ as $n \to +\infty$.

By definition, any $\alpha_{n}$ can be written as $ \sum_{i \in I} \lambda_{n,i} \alpha_{i}$ with $\lambda_{n,i} \in \Z_{+}$ for all $(n,i) \in \Z_{+} \times I$. The injectivity of the sequence $(\alpha_{n})_{n \geq 0}$ implies that $ \sum_{i\in I} \lambda_{n,i}\to +\infty$ as $n$ goes to $+\infty$, hence $ \lim_{n \to +\infty}\alpha_{n}(a) = +\infty$ as required.
\end{proof}

\begin{cor}
\label{CorParahoriques non ouverts compacts}
Let $F$ be a type $0$ face of $\I$. If $W^{v}$ is infinite, then there is no topology of topological group on $G$ for which $K_{F}$ is open and compact.
\end{cor}
\begin{proof}
If there was such a topology, then for any $g \in G$, $K_{F}$ and $K_{g\cdot F} = gK_{F}g^{-1}$ would be open and compact in $G$, hence $K_{F} \cap K_{g\cdot F}$ would have the same properties. This would imply the finiteness of the quotient $K_{F}/ K_{F} \cap K_{g\cdot F}$ for any $g \in G$, which contradicts Lemma~\ref{lemNonCommensurabiliteParahoriques}.
\end{proof}

Considering $F = F_{0}$ (resp. $F = C^{+}_{0}$), we obtain that $K_{s}$ (resp. $K_{I}$) cannot be open and compact in $G$ when $W^{v}$ is infinite, \textit{i.e} when $G$ is not reductive. This shows how different reductive groups and (non-reductive) Kac-Moody groups are from this point of view. 

\section{The completed Iwahori-Hecke algebra}
\label{secAlgebre d'IH completee}
\subsection{Definition of the usual Iwahori-Hecke algebra}
\label{subsecIH algebre}
Let us first recall briefly the construction of the Iwahori-Hecke algebra via its Bernstein-Lusztig presentation, as done in \cite[\S 6.6]{bardy2016iwahori}. Note that this definition requires some restrictions on the possible choices for the ring of scalars; nevertheless, choosing $\C$ or $\Z[\sqrt{q}, \sqrt{q}^{-1}]$ is allowed when $G$ is a split Kac-Moody group over $\mathcal{K}$. Another definition of the Iwahori-Hecke algebra (as an algebra of functions on pairs of type $0$ chambers in a masure) is given in \cite[Def.\,2.5]{bardy2016iwahori} and allows more flexibility in the choice of scalars. This will be recalled in Section~\ref{secAlgebres de Hecke generales}.

Let $\mathcal{R}_{1}:= \Z[(\sigma_{i}, \sigma'_{i})_{i \in I}]$, where $(\sigma_i)_{i\in I}$ and $(\sigma'_i)_{i\in I}$ denote indeterminates that satisfy the following relations:
\begin{itemize}
\item if $\alpha_{i}(Y) = \Z$, then $\sigma_{i} = \sigma'_{i}$;
\item if $(i,j) \in I^{2}$ are such that $r_{i}$ and $r_{j}$ are conjugate (\textit{i.e} such that $\alpha_{i}(\alpha^{\vee}_{j}) = \alpha_{j}(\alpha^{\vee}_{i}) = -1$), then $\sigma_{i}=\sigma_{j}=\sigma'_{i}=\sigma'_{j}$.
\end{itemize}

To define the Iwahori-Hecke algebra $\HH$ associated with $\A$ and $(\sigma_{i}, \sigma'_{i})_{i \in I}$, we first introduce the Bernstein-Lusztig-Hecke algebra. Let $\BLHH$ be the free $\RR_{1}$-module with basis $(Z^\lambda H_w)_{\lambda\in Y,w\in W^v}$. For short, set $H_{i}:= H_{r_{i}}$ for $i \in I$, as well as $H_{w} = Z^{0}H_{w}$ for $w \in W^{v}$ and $Z^{\lambda} = Z^{\lambda}H_{1}$ for $\lambda \in Y$. The \textbf{Bernstein-Lusztig-Hecke algebra} $\BLHH$ is the module $\BLHH$ equipped with the unique product $*$ that turns it into an associative algebra and satisfies the following relations (known as \textbf{Bernstein-Lusztig relations}):
\begin{align*}
\tag{BL1}
&\forall (\lambda, w) \in Y \times W^{v},\quad Z^{\lambda} * H_{w} = Z^{\lambda}H_{w};\\
\tag{BL2}
&\forall i \in I,\, \forall w \in W^{v},\
H_{i}*H_{w}=\begin{cases}
H_{r_{i}w} &\hspace*{-2mm}\text{if }\ell(r_{i}w)=\ell(w)+1,\\
(\sigma_{i}-\sigma_{i}^{-1}) H_{w}+H_{r_{i} w} &\hspace*{-2mm}\text{if }\ell(r_iw)=\ell(w)-1;
\end{cases}\\
\tag{BL3}
&\forall (\lambda, \mu) \in Y^{2},\quad Z^{\lambda} * Z^{\mu} = Z^{\lambda + \mu};\\
\tag{BL4}
&\forall \lambda \in Y, \, \forall i \in I,\quad H_{i}*Z^{\lambda}-Z^{r_{i}(\lambda)}*H_{i} = b(\sigma_{i},\sigma'_{i};Z^{-\alpha^{\vee}_{i}})(Z^{\lambda}-Z^{r_{i}(\lambda)}),\\
&\text{with}\quad b(t,u;z)= \frac{(t-t^{-1})+(u-u^{-1}) z}{1-z^2}.
\end{align*}
The existence and unicity of such a product $*$ comes from \cite[Th.\,6.2]{bardy2016iwahori}. Following \cite[\S 6.6]{bardy2016iwahori}, the \textbf{Iwahori-Hecke algebra} $\HH_{\RR_{1}}$ associated with $\A$ and $(\sigma_{i}, \sigma'_{i})_{i \in I}$ is now defined as the $\RR_{1}$-submodule of $\BLHH$ spanned by $(Z^{\lambda} H_{w})_{\lambda\in Y^{+}, w\in W^{v}}$ (recall that $Y^{+} = Y \cap \T$ with $\T$ being the Tits cone). Note that for $G$ reductive, we recover the usual Iwahori-Hecke algebra of $G$.

\begin{rque}
\label{remExtension des scalaires}
This construction is compatible with extension of scalars. Let indeed~$\RR$ be a ring that contains $\Z$ and $\phi: \RR_{1} \to \RR$ be a ring homomorphism such that $\phi(\sigma_{i})$ and $\phi(\sigma'_{i})$ are invertible in $\RR$ for all $i \in I$: then the Iwahori-Hecke algebra associated with $\A$ and $(\phi(\sigma_{i}), \phi(\sigma'_{i}))_{i \in I}$ over $\RR$ is $ \HH_{\RR} = \RR \otimes_{\RR_{1}} \HH_{\RR_{1}}$.
\end{rque}
\begin{rque}
\label{remIH algebre dans le cas KM deploye}
When $G$ is a split Kac-Moody group over $\mathcal{K}$, we can (and will) set $\sigma_{i} = \sigma'_{i} = \sqrt{q}$ for all $i \in I$ and $\RR = \Z[\sqrt{q}{}^{\pm1}]$. The corresponding Iwahori-Hecke algebra $\HH_{\RR}$ will simply be denoted by $\HH$.
\end{rque}

\subsection{Almost finite sets in $Y$ and $Y^{+}$}
We fix a pair $(\RR, \phi)$ as in Remarks \ref{remExtension des scalaires} and~\ref{remIH algebre dans le cas KM deploye}. In this section, we introduce a notion of almost finite sets in $Y$ and $Y^{+}$, that will be used to define the Looijenga algebra $\RR\llbracket Y\rrbracket$ in the next section and the completed Iwahori-Hecke algebra $\hat{\HH} = \hat{\HH}_{\RR}$ in Section~\ref{subecCompleted IH}.

\subsubsection{Definition of almost finite sets}
\begin{defn}
\label{Def partie presque finie}
A subset $E$ of $Y$ is \textbf{almost finite (in $Y$)} if there is a finite set $J \subset Y$ such that: $\forall \lambda \in E, \ \exists \nu \in j \mid \lambda \leq_{Q^{\vee}} \nu$.
\end{defn}
Replacing $Y$ by $Y^{+}$ in the previous definition, we have the definition of almost finite sets in $Y^{+}$. Nevertheless, the following lemma (applied to $F = Y^{+}$) justifies why we do not set this other definition apart: it shows indeed that almost finiteness for $Y^{+}$ can already be seen in $Y$, which explains why we will just write \textbf{almost finite sets} with no more specification.

\begin{lemme}
\label{lem intersection ensembles presque finis}
Let $E \subset Y$ be an almost finite set. For any subset $F$ of $Y$, there exists a finite set $J \subset F$ such that $F \cap E \subset \bigcup_{j \in J} (j - Q^{\vee}_{+})$.
\end{lemme}
\begin{proof}
As $E$ is almost finite, we can assume that $E$ is contained in $y - Q^{\vee}_{+}$ for some well-chosen $y \in Y$. Let $J$ be the set of all elements in $F \cap E$ that are maximal in $F \cap E$ for the pre-order $\leq_{Q^{\vee}}$. As $E$ is almost finite, we already have: $\forall x \in E, \ \exists \nu \in\nobreak K \mid x \leq_{Q^{\vee}} \nu$. Let us prove that $J$ is finite, which will conclude the proof. To~do this, we identify $Q^\vee$ with $\Z^{I}$ and set $J':= \{u \in Q^\vee \mid y - u \in J\}$. We define a comparison relation $\prec$ on $Q^{\vee}$ as follows: for all $x = (x_{i})_{i \in I}$ and $x' = (x'_{i})_{i \in I}$, we write $x \prec x'$ when $x_{i} \leq x'_{i}$ (in $\Z$) for all $i \in I$ and $x \not= x'$. By definition of $J$, elements of $J'$ are pairwise non comparable, hence \cite[Lem.\,2.2]{hebertGK} implies that $J'$ is finite, which requires that $J$ itself is finite and completes the proof.
\end{proof}

\subsubsection{Examples of almost finite sets in $Y^{+}$}
\subsubsection*{In the affine case}
Suppose that $\A$ is associated with an affine Kac-Moody matrix~$A$. By \cite[Rem.\,5.10]{rousseau2011masures}, we know that $\T = \delta^{-1}(\R^{*}_{+}) \sqcup \A_\mathrm{in}$, where $\delta$ denotes the smallest positive imaginary root of $A$, and that $\delta$ is $W^{v}$-invariant, thus $\delta(\alpha^{\vee}_{i}) >0$ for all $i \in I$. Therefore, an almost finite set of $Y^{+}$ is a set $E$ contained in $ \bigcup_{i = 1}^{k} (y_{i} - Q^{\vee}_{+})$ for some integer $k \geq 1$ and some $y_{1}, \ldots, y_{k} \in Y^{+}$.

\subsubsection*{In the indefinite case}
Unlike the finite or the affine case, when $\A$ is associated with an indefinite Kac-Moody matrix $A$, we have: $\forall y \in Y, \ y - Q^{\vee}_{+} \not\subseteq Y^{+}$. Indeed, due to the proof and the statement of \cite[Lem.\,2.9]{gaussent2014spherical}, there exists a linear form $\delta: \A \to \R$ such that $\delta(\T) \geq 0$ and $\delta(\alpha^{\vee}_{i}) < 0$ for all $i \in I$. Consequently, if $y \in Y$ and $i \in I$, then~$\delta(y - n\alpha^{\vee}_{i}) < 0$ for $n$ large enough, hence $y - Q^{\vee}_{+}$ is not contained in $Y^{+}$. However, $Y^{+}$ may be contained in $Q^{\vee}_{-}$, as stated by the following lemma.
\begin{lemme}
\label{lem equivalence inclusions cas indefini}
We have $Y^{+} \subset Q^{\vee}_{-}$ iff we have $Y^{++} \subset Q^{\vee}_{-}$.
\end{lemme}
\begin{proof}
As $Y^{++}$ is contained in $Y^{+}$, the direct implication is obvious. Assume conversely that $Y^{++} \subset Q^{\vee}_{-}$. By Lemma~\ref{lemLemme 2.4 a) de GR14}, we know that $\lambda \leq_{Q^{\vee}} \lambda^{++}$ for any $\lambda \in Y^{+}$, with $\lambda^{++}$ in $Y^{++}$, hence we must have $Y^{+} \subset Q^{\vee}_{-}$ and the proof is complete.
\end{proof}

We say that $\A$ is the \textbf{essential realization of the Kac-Moody matrix $A$} when $\A_\mathrm{in} =\nobreak \{0\}$, or equivalently when $\dim_{\R}\A$ equals the size of the matrix $A$. If $\A$ is the essential realization of an indefinite matrix $A = \begin{pmatrix}
2& a_{1,2}\\ a_{2,1} & 2
\end{pmatrix}$ of size $2$, with $a_{1,2}$ and~$a_{2,1}$ negative integers, then $\A = \R\alpha^{\vee}_{1} \oplus \R\alpha^{\vee}_{2}$. For any integers $\lambda$ and $\mu$ of opposite sign (\textit{i.e} such that $\lambda\mu < 0$), we have $(2\lambda +a_{1,2} \mu)(a_{2,1}\lambda+2\mu)<0$, hence $Y^{++}$ is contained in $Q^{\vee}_{+} \cup Q^{\vee}_{-}$. By \cite[Th.\,4.3]{kac1994infinite}, we get that $Q^{\vee}_{+} \cap Y^{++} = \{0\}$, hence $Y^{++}$ is contained in $Q^{\vee}_{-}$, and Lemma \ref{lem equivalence inclusions cas indefini} implies that $Y^{+}$ is also contained in $Q^{\vee}_{-}$. Consequently, every subset of $Y^{+}$ is almost finite.

Note that this conclusion does not always hold when $A$ is of size $n \geq 3$. Indeed, assume for instance that $\A$ is the essential realization of the matrix
\[ A = \begin{pmatrix}
2 & 0 & -2\\ 0 & 2 & 0 \\ -5 & 0 & 2
\end{pmatrix}.\]
Then $-2\alpha_1^\vee+\alpha_2^\vee-\alpha_3^\vee$ is in $Y^{++}$ but not in $Q^{\vee}_{-}$.

\subsection{The Looijenga algebra $\RR\llbracket Y\rrbracket$}
We keep the previous notations. The next definition follows the definition of the algebra $A$ given in \cite[\S 4]{looijenga1980invariant}.
\begin{defn}
\label{DefLooijenga algebra}
Let $(e_{\lambda})_{\lambda \in Y}$ be a family of symbols that satisfy $e^{\lambda} e^{\mu} =e^{\lambda+\mu}$ for all $\lambda,\mu \in Y$. The \textbf{Looijenga algebra} $\RR\llbracket Y\rrbracket$ of $Y$ over $\RR$ is defined as the set of formal series $ \sum_{\lambda\in Y} a_{\lambda} e^{\lambda}$ with $(a_{\lambda})_{\lambda \in Y} \in \RR^{Y}$ having almost finite support.
\end{defn}

For any element $\lambda \in Y$, let $\pi_{\lambda}: \RR\llbracket Y\rrbracket \to \RR$ be the ``$\lambda$-the coordinate map'' defined by $ \pi_{\lambda}\bigl(\sum_{\mu \in Y} a_{\mu}e^{\mu}\bigr):= a_{\lambda}$. Define $\RR\llbracket Y^{+}\rrbracket$ and $\RR\llbracket Y\rrbracket^{W^{v}}$ as follows:
\[
\begin{cases}
\RR\llbracket Y^{+}\rrbracket:= \left\{a \in \RR\llbracket Y\rrbracket \mid \forall y \in Y \setminus Y^{+}, \pi_{\lambda}(a) = 0\right\} \\[3pt]
\RR\llbracket Y\rrbracket^{W^{v}}:= \left\{ a \in \RR\llbracket Y\rrbracket \mid \forall (\lambda, w) \in Y \times W^{v}, \ \pi_{\lambda}(w\cdot a) = \pi_{\lambda}(a) \right\}.
\end{cases}
\]
One can check that $\RR\llbracket Y^{+}\rrbracket$ and $\RR\llbracket Y\rrbracket^{W^{v}}$ are $\RR$-subalgebras of $\RR\llbracket Y\rrbracket$.
\begin{defn}
\label{def Familles sommables Looijenga algebra}
A family $(a_{j})_{j \in J} \in (\RR\llbracket Y\rrbracket)^{J}$ is \textbf{summable} if:
\begin{itemize}
\item for all $\lambda \in Y$, $\{ j \in J \mid \pi_{\lambda}(a_{j})\not= 0\}$ is finite;
\item the set $\{\lambda \in Y \mid \exists j \in J,\, \pi_{\lambda}(a_{j}) \not= 0\}$ is almost finite.
\end{itemize}
\end{defn}
Given a summable family $(a_{j})_{j \in J} \in (\RR\llbracket Y\rrbracket)^{J}$, we set $ \sum_{j \in J} a_{j}:= \sum_{\lambda \in Y} b_{\lambda} e^{\lambda}$, with $b_{\lambda}:= \sum_{j \in J} \pi_{\lambda}(a_{j})$ for any $\lambda \in Y$. For $\lambda \in Y^{++}$, set $E(\lambda):= \sum_{\mu \in W^{v}\cdot\lambda} e^{\mu} \in \RR\llbracket Y\rrbracket$. (Note that this is well-defined by Lemma~\ref{lemLemme 2.4 a) de GR14}.) Finally, for any $\lambda \in \mathcal{T}$, let $\lambda^{++}$ be the unique element in $\overline{C^{v}_{f}}$ that has the same $W^{v}$-orbit as $\lambda$ (\textit{i.e} such that $W^{v}\cdot\lambda = W^{v}\cdot\lambda^{++}$).

\begin{lemme}
\label{lemSupport inclus dans Y++}
Let $y \in Y$. Then $W^{v}\cdot y$ is upper-bounded (for $\leq_{Q^{\vee}}$) iff $y$ belongs to~$Y^{+}$.
\end{lemme}
\begin{proof}
If $y$ belongs to $Y^{+}$, then Lemma~\ref{lemLemme 2.4 a) de GR14} implies that $W^{v}\cdot y$ is upper-bounded by $y^{++}$. Assume conversely that $y \in Y$ is such that $W^{v}\cdot y$ is upper-bounded for $\leq_{Q^{\vee}}$ and let $x \in W^{v}\cdot y$ be a maximal element. For any $i \in I$, we have $r_{i}(x) \leq_{Q^{\vee}} x$, hence $\alpha_{i}(x) \geq 0$, which proves that $x$ belongs to $\overline{C^{v}_{f}}$ and implies that $y$ is in $Y^{+}$.
\end{proof}

Denote by $\mathrm{AF}_{\RR}(Y^{++})$ the set of elements of $\RR^{Y^{++}}$ having almost finite support.
\begin{prop}
\label{propR[[Y]] Wv inclus dans R[[ Y+]]}
The map $E: \mathrm{AF}_{\RR}(Y^{++}) \to \RR\llbracket Y\rrbracket^{W^{v}}$ that sends $a \in \mathrm{AF}_{\RR}(Y^{++})$ to $ \sum_{\lambda \in Y^{++}} \pi_{\lambda}(a) E(\lambda)$ is well-defined and bijective. In particular, we have $\RR\llbracket Y\rrbracket^{W^{v}} \subset \RR\llbracket Y^{+}\rrbracket$.

\end{prop}
\begin{proof}
Let $a = (a_{\lambda})_{\lambda \in Y^{++}}$ be an element of $\mathrm{AF}_{\RR}(Y^{++})$. As $a$ has almost finite support, there exists a finite set $J \subset Y$ such that: $\forall \lambda \in \supp(a), \ \exists \nu \in J \mid \lambda \leq_{Q^{\vee}} \nu$. We start by proving that $(a_{\lambda}E(\lambda))_{\lambda \in Y^{++}}$ is summable. Let $\nu \in Y$ and set
\[
F_{\nu}:= \{ \lambda \in Y^{++} \mid \pi_{\nu}(a_{\lambda}E(\lambda)) \neq 0\}.
\]
For any $\lambda \in F_{\nu}$, $\nu$ belongs to $W^{v}\cdot\lambda$, hence Lemma~\ref{lemLemme 2.4 a) de GR14} implies that $\nu \leq_{Q^{\vee}} \lambda$. As there exists moreover some $j \in J$ such that $\lambda \leq_{Q^{\vee}} j$, we get the finiteness of $F_{\nu}$. Now let $F:= \bigcup_{\nu \in Y} F_{\nu}$. We just saw that any element of $F$ is dominated (for $\leq_{Q^{\vee}}$) by some element of $J$, hence $F$ is by definition almost finite, and $(a_{\lambda}E(\lambda))_{\lambda \in Y^{++}}$ is summable.

By construction, $E(\lambda)$ is in $\RR\llbracket Y\rrbracket^{W^{v}}$ for any $\lambda \in Y^{++}$, hence $ \sum_{\lambda \in Y^{++}} a_{\lambda}E(\lambda)$ is in $\RR\llbracket Y\rrbracket^{W^{v}}$ too and $E$ is well-defined. Now assume that $a \in \mathrm{AF}_{\RR}(Y^{++})$ is non-zero and let $\nu$ be maximal (for $\leq_{Q^{\vee}}$) among the elements $\lambda$ of $Y^{++}$ such that $\pi_{\lambda}(a) \not= 0$. Then $\pi_{\nu}(E(a)) = \pi_{\nu}(a) \not= 0$, hence $E(a)$ is non-zero and $E$ is injective. To prove $E$ is surjective, let $u = \sum_{\lambda \in Y} u_{\lambda} e^{\lambda}$ be any element of $\RR\llbracket Y\rrbracket^{W^{v}}$ and let $\lambda \in \supp(u)$. As $\supp(u)$ is almost finite and $W^v$-invariant, $W^{v}\cdot\lambda$ is upper-bounded, hence Lemma~\ref{lemSupport inclus dans Y++} implies that $\lambda$ belongs to $Y^{+}$. This proves that $\supp(u)$ is contained in $Y^{+}$, and that $u = E((\pi_{\lambda}(u))_{\lambda \in \supp(u)^{++}})$ is in the image of $E$, which completes the proof.
\end{proof}

\begin{rque}
Lemma \ref{lemSupport inclus dans Y++} and Proposition \ref{propR[[Y]] Wv inclus dans R[[ Y+]]} are not explicitly stated in \cite{gaussent2014spherical}, but their proof is basically contained in the proof of \cite[Th.\,5.4]{gaussent2014spherical}.
\end{rque}

\subsection{The completed Iwahori-Hecke algebra $\hat{\HH}$}
\label{subecCompleted IH}
In this subsection, we define an $\RR$-algebra $\hat{\HH}$ as a ``completion'' of the usual Iwahori-Hecke algebra, what justifies the name of \textbf{completed Iwahori-Hecke algebra} given to $\hat{\HH}$. In the next section, we will compute the centers of $\HH$ and of $\hat{\HH}$, and recover the reasons that motivated the introduction of $\hat{\HH}$ in this context.

Endow $W^{v}$ with its Bruhat order $\leq$ and, for any $w \in W^{v}$, set
\[
[1,w]:= \{u \in W^{v} \mid u \leq w\}.
\]
This notation makes sense as $1 \leq w$ for all $w \in W^{v}$. Let $\mathcal{B}:= \prod_{\lambda \in Y^{+}, \ w \in W^{v}} \RR$.
\begin{defn}
\label{def differents supports}
For any $a = (a_{\lambda, w})_{(\lambda, w) \in Y^{+} \times W^{v}}$ in $\mathcal{B}$, the \textbf{support of $a$} is the set
\[ \supp(a):= \left\{ (\lambda,w) \in Y^{+} \times W^{v} \mid a_{\lambda, w} \not= 0 \right\}.
\]
The \textbf{support of $a$ along $Y$} is the set
\[ \supp_{Y}(a):= \left\{ \lambda \in Y^{+} \mid \exists w \in W^{v}, a_{\lambda, w} \not= 0 \right\}
\]
and the \textbf{support of $a$ along $W^{v}$} is the set
\[ \supp_{W^{v}}(a):= \left\{ w \in W^{v} \mid \exists \lambda \in Y^{+}, a_{\lambda, w} \not= 0 \right\}.\]
\end{defn}

\begin{defn}
\label{def presque fini dans le produit Y+ times Wv}
A subset $Z$ of $Y^{+} \times W^{v}$ is \textbf{almost finite} if $$\{w \in W^{v} \mid \exists \lambda \in Y^{+}, (\lambda, w) \in Z\}$$ is finite and if, for all $w \in W^{v}$, the set $\{\lambda \in Y^{+} \ \vert (\lambda, w) \in Z\}$ is almost finite (in the sense of Definition \ref{Def partie presque finie}).
\end{defn}

Let $\hat{\HH}$ be the set of all elements in $\mathcal{B}$ with almost finite support. An element $(a_{\lambda, w})_{(\lambda, w) \in Y^{+} \times W^{v}}$ of $\hat{\HH}$ will also be written as $ \sum_{(\lambda, w) \in Y^{+} \times W^{v}} a_{\lambda,w} Z^{\lambda}H_{w}$. Any pair $(\lambda, w) \in Y^{+} \times W^{v}$ defines a projection map $\pi_{\lambda, w}: \hat{\HH} \to \RR$ defined by $ \pi_{\lambda, w}\bigl(\sum_{(\nu, u) \in Y^{+} \times W^{v}} a_{\nu,u} Z^{\nu}H_{u} \bigr):= a_{\lambda, w}$.

To extend the product $*$ to $\hat{\HH}$, we start by proving that for any elements
\[
\sum_{(\lambda, w) \in Y^{+} \times W^{v}} a_{\lambda,w} Z^{\lambda}H_{w}\quad\text{and}\quad \sum_{(\lambda, w) \in Y^{+} \times W^{v}} b_{\lambda,w} Z^{\lambda}H_{w}
\]
of $\hat{\HH}$, and any pair $(\mu, v) \in Y^{+} \times W^{v}$, the sum
\[\sum_{(\lambda,w), (\lambda',w')\in Y^+\times W^v}\pi_{\mu,v} (a_{\lambda,w} b_{\lambda',w'} Z^\lambda H_w * Z^{\lambda'} H_{w'})\]
is a finite sum, \textit{i.e} that only finitely many terms $\pi_{\mu,v} (a_{\lambda,w} b_{\lambda',w'} Z^\lambda H_w * Z^{\lambda'} H_{w'})$ are non-zero. The key fact to prove this is that for any pair $(\lambda, w) \in Y \times W^{v}$, the support of $H_{w}*Z_{\lambda}$ along $Y^{+}$ is in the convex hull of $\{ u\cdot \lambda, u \in [1,w]\}$. This fact comes from Lemma~\ref{lem Reecriture relations commutation longueur qcq} below. 

For any subset $E$ of $Y$ and any $i \in I$, set $R_{i}(E):= \conv(\{E, r_{i}(E)\}) \subset E + Q^{\vee}$. When $E = \{\lambda\}$ is reduced to a single element, we write $R_{i}(\lambda)$ instead of $R_{i}(\{\lambda\})$. For any $w \in W^{v}$ and any $\lambda \in Y$, set $R_{w}(\lambda):= \bigcup R_{i_{1}}(R_{i_{2}}(\ldots (R_{i_{k}}(\lambda)) \ldots))$, where the union is taken over all the reduced writings $r_{i_{1}}.r_{i_{2}}\ldots r_{i_{k}}$ of $w$.
\begin{rque}
\label{rqueFinitude des Rw(lambda)}
For any pair $(\lambda,w) \in Y \times W^{v}$, the set $R_{w}(\lambda)$ is actually finite. Indeed, given any finite set $E$ and any $i \in I$, the set $R_{i}(E)$ is bounded and contained in $E + Q^{\vee}$, hence must be finite. By induction, we get that for any integer $k \geq 0$ and any list $(i_{1}, \ldots, i_{k})$ of elements of $I$, the set $R_{i_1}(R_{i_2}(\ldots (R_{i_k}(E))\ldots))$ is also finite. As $w$ has only finitely many reduced writings,\footnote{The number of reduced writings of $w$ is upper-bounded by $|I|^{\ell(w)}$.} we obtain that $R_{w}(\lambda)$ is finite.
\end{rque}

\begin{lemme}
\label{lemReecriture relations commutation longueur 1}
For all $i \in I$ and all $\lambda \in Y$, the product $H_{i} * Z^{\lambda}$ is in $$ \bigoplus_{(\nu, t) \in R_{i}(\lambda)\times \{1,r_{i}\}} \RR\cdot Z^{\nu}H_{t}.$$
\end{lemme}
\begin{proof}
Let $i \in I$ and $\lambda \in Y$. If $\sigma_{i} = \sigma'_{i}$, then (BL4) implies that
\[ H_{i}*Z^{\lambda} = Z^{r_{i}(\lambda)}*H_{i} + (\sigma_{i} - \sigma_{i}^{-1})Z^{\lambda}\, \frac{1-Z^{-\alpha_{i}(\lambda)\alpha^{\vee}_{i}}}{1 - Z^{-\alpha^{\vee}_{i}}},
\]
hence we have the following alternative.
\begin{itemize}
\item Either $\alpha_{i}(\lambda) = 0$, and $H_{i}*Z^{\lambda} = Z^{\lambda}*H_{i}$, \textit{i.e} $H_{i}$ and $Z^{\lambda}$ commute to each other.
\item Or $\alpha_{i}(\lambda) >0$, in which case we have
\[ H_{i}* Z^{\lambda} = Z^{r_{i}(\lambda)}*H_{i} + (\sigma_{i} - \sigma_{i}^{-1})\sum_{k = 0}^{\alpha_{i}(\lambda) - 1} Z^{\lambda - k\alpha^{\vee}_{i}}\]
with $r_{i}(\lambda)$ and $\lambda - k\alpha^{\vee}_{i}$ in $R_{i}(\lambda)$ for all $k \in \llbracket 0, \alpha_{i}(\lambda) - 1 \rrbracket$.
\item Or $\alpha_{i}(\lambda) < 0$, in which case we have
\[ H_{i}*Z^{\lambda} = Z^{r_{i}(\lambda)}*H_{i} + (\sigma_{i} - \sigma_{i}^{-1})\sum_{k = 1}^{-\alpha_{i}(\lambda)} Z^{\lambda + k\alpha^{\vee}_{i}}\]
with $r_{i}(\lambda)$ and $\lambda + k\alpha^{\vee}_{i}$ in $R_{i}(\lambda)$ for all $k \in \llbracket 1, -\alpha_{i}(\lambda) \rrbracket$.
\end{itemize}
In any case, we have proven that $H_{i}* Z^{\lambda}$ is in $ \bigoplus_{(\nu, t) \in R_{i}(\lambda) \times \{1,r_{i}\}} \RR Z^{\nu}H_{t}$ when $\sigma_{i} = \sigma'_{i}$.

If $\sigma_{i} \not= \sigma'_{i}$, then $\alpha_{i}(Y) = 2\Z$, so we have now
\[ H_{i} * Z^{\lambda} = Z^{r_{i}(\lambda)} * H_{i} + Z^{\lambda} \left((\sigma_{i} - \sigma^{-1}_{i}) + (\sigma'_{i} - \sigma'^{-1}_{i})\right) Z^{-\alpha^{\vee}_{i}},
\]
and similar computations as those done in the $\sigma_{i} = \sigma'_{i}$ case complete the proof.
\end{proof}

\begin{lemme}
\label{lem Reecriture relations commutation longueur qcq}
For all $w,w' \in W^{v}$ and all $\lambda \in Y$, $H_{w'}*Z^{\lambda}H_{w}$ is in $$ \bigoplus_{(\nu, t) \in R_{w}(\lambda) \times [1,w]\cdot w'} \RR\cdot Z^{\nu}H_{t}.$$
\end{lemme}
\begin{proof}
The proof goes by induction on $\ell(w') \geq 0$. There is nothing to prove if $\ell(w') = 0$, and if $\ell(w') = 1$, this is exactly Lemma \ref{lemReecriture relations commutation longueur 1}. Now, fix an integer $k \geq 1$ and a pair $(w, \lambda) \in W^{v}\times Y^{+}$, and assume that for any $u \in W^{v}$ satisfying $\ell(u) \leq k-1$, the product $H_{u}*Z^{\lambda}H_{w}$ belongs to $ \bigoplus_{(\nu, t) \in R_{w}(\lambda) \times [1,w]\cdot u} \RR\cdot Z^{\nu}H_{t}$. Let $w' \in W^{v}$ be an element of length $k$ and write $w' = r_{i}u$ for $i \in I$ and $u \in W^{v}$ of length $k-1$. Then $ H_{w'}*Z^{\lambda}H_{w} = H_{i} * H_{u}*Z^{\lambda}H_{w}$ belongs to $ \bigoplus_{(\nu, t) \in R_{u}(\lambda) \times [1,u]\cdot w} \RR H_{i} * Z^{\lambda}H_{t}$, with
\[ \bigoplus_{(\nu, t) \in R_{u}(\lambda) \times [1,u]\cdot w} \RR H_{i} * Z^{\lambda}H_{t} \subset \bigoplus_{(\nu, t) \in R_{u}(\lambda) \times [1,u]\cdot w} \bigoplus_{(\nu', t') \in R_{i}(\nu) \times \{1, r_{i}\}} \RR\cdot Z^{\nu'}*H_{t'}*H_{t}\]
by Lemma \ref{lemReecriture relations commutation longueur 1}. Using (BL2) for $t' \in \{1,r_{i}\}$ and $t \in [1,u]\cdot w$, we get that $H_{t'}*H_{t}$ belongs to $\RR H_{r_{i}t} \oplus \RR H_{t}$, hence to $ \bigoplus_{v \in [1,w']\cdot w} \RR H_{v}$, and the lemma follows.
\end{proof}

\begin{lemme}
\label{lemEcriture des elements de Rwlambda}
Let $(\lambda, w) \in Y \times W^{v}$. For any $\nu \in R_{w}(\lambda)$, there is a family $(a_{u})_{u \leq w} \in [0,1]^{[1,w]}$ such that $ \sum_{u \in [1,w]} a_{u} = 1$ and $\nu = \sum_{u \in [1,w]} a_{u}\cdot u\cdot \lambda$.
\end{lemme}
\begin{proof}
This proof goes again by induction on $\ell(w) \geq 0$. There is nothing to do when $\ell(w) = 0$. Let $k \in \Z$ be non-negative and assume that the statement of the lemma is true whenever $w \in W^{v}$ is of length $k$. Let $w' \in W^{v}$ be of length $k+1$, fix $\lambda \in Y$ and pick $\nu \in R_{w'}(\lambda)$. Then there exists some triple $(i, w, \nu') \in I \times W^{v} \times W^{v}$, with $w$ of length $k$ and $\nu' \in R_{w}(\lambda)$, such that $\nu \in R_{i}(\nu')$. As we can write $\nu' = \sum_{u \in [1,w]} a_{u}u\cdot \lambda$ for some family $(a_{u})_{u \in [1,w]} \in [0,1]^{[1,w]}$ and $\nu = s\nu' + (1-s)r_{i}\nu'$ for some $s \in [0,1]$, we finally get that
$ \nu = \sum_{u \in [1,w]} \left(sa_{u}u\cdot \lambda + (1-s)a_{u}r_{i}u\cdot \lambda\right)$ with $r_{i}u \in [1,w']$ for all $u \in [1,w]$, hence the lemma follows.
\end{proof}

\begin{lemme}
\label{lemPositivite de Rwlambda inferieur a lambda}
\begin{enumerate}
\item For all $\lambda, \mu \in Y^{+}$, we have $(\lambda + \mu)^{++} \leq_{Q^{\vee}} \lambda^{++} + \mu^{++}$.
\item Let $(\lambda, w) \in Y^{+} \times W^{v}$. For all $\nu \in R_{w}(\lambda)$, we have $\nu^{++} \leq_{Q^{\vee}} \lambda^{++}$.
\end{enumerate}
\end{lemme}
\begin{proof}
Let $\lambda, \mu \in Y^{+}$ and $w \in W^{v}$ be such that $(\lambda + \mu)^{++} = w\cdot (\lambda + \mu)$. By Lemma~\ref{lemLemme 2.4 a) de GR14}, we have $w\cdot \lambda \leq_{Q^{\vee}} \lambda^{++}$ and $w\cdot \mu \leq_{Q^{\vee}} \mu^{++}$, hence we get the first statement. Together with Lemma~\ref{lemLemme 2.4 a) de GR14}, Remark~\ref{rqueFinitude des Rw(lambda)} and Lemma~\ref{lemEcriture des elements de Rwlambda}, the first statement implies the second one.
\end{proof}

Define the \textbf{height} $h(x)$ of any $x = \sum_{i \in I} x_{i} \alpha^{\vee}_{i} \in Q^{\vee}$ by $h(x):= \sum_{i \in I} x_{i}$. For any $\lambda \in Y^{+}$, let $w_{\lambda} \in W^{v}$ be the element of minimal length such that $w^{-1}\cdot \lambda \in \overline{C^{v}_{f}}$: then we have $\lambda = w_{\lambda}\cdot\lambda^{++}$.
\begin{lemme}
\label{lemComportement asymptotique hauteur longueur}
Let $\lambda \in Y^{++}$ and $(\mu_{n})_{n \geq 0} \in (W^{v}\cdot\lambda)^{\Z_{+}}$ be such that $$ \lim_{n \to +\infty} \ell(w_{\mu_{n}}) = + \infty.$$ Then $ \lim_{n \to +\infty} h(\mu_{n} - \lambda) = -\infty$.
\end{lemme}
\begin{proof}
By Lemma~\ref{lemLemme 2.4 a) de GR14}, we know that $h(\mu_{n} - \lambda)$ is well-defined. For all $\alpha \in \Phi_{+}$, we have $\alpha(\lambda) \geq 0$. Assume that $\alpha(\lambda) = 0$: then $\alpha_{i}(\lambda) = 0$ for all $i \in I$, hence $r_{i}\lambda = \lambda$ for all $i \in I$ and $W^{v}\cdot\lambda$ is then reduced to $\{\lambda\}$, which contradicts the fact that $ \lim_{n \to +\infty} \ell(w_{\mu_{n}}) = + \infty$. This proves that $\alpha(\lambda) > 0$ for all $\alpha \in \Phi_{+}$.

Now let $\rho: \A \to \R$ be a linear form satisfying $\rho(\alpha^{\vee}_{i}) = 1$ for all $i \in I$. Pick $n \in \Z_{+}$ and set $w:=w_{\mu_{n}}$. Then we have
$$ h(\mu_{n} - \lambda) = h(w\cdot \lambda - \lambda) = \rho(w\cdot \lambda - \lambda) = (w^{-1}\rho - \rho)(\lambda).$$ Thanks to \cite[Cor.\,(3)]{kumar2002kac}, we know that $w^{-1}\cdot \rho - \rho = - \sum_{\alpha \in \Phi_{w^{-1}}} \alpha$ with $\Phi_{w^{-1}}:= \Delta^{+} \cap w^{-1}\cdot \Delta^{-}$. By \cite[Lem.\,3.14]{kumar2002kac}, we also know that $\vert \Phi_{w^{-1}} \vert = \ell(w^{-1}) = \ell(w)$. Letting $n$ go to $+\infty$, we obtain that $$h(\mu_{n} - \lambda) = - \sum_{\alpha \in \Phi_{w_{\mu_{n}}^{-1}}} \alpha(\lambda)$$ goes to $-\infty$, as required.
\end{proof}

\begin{lemme}
\label{lemFinitude des antecedants de mu par Rw}
For all $(\lambda, \mu, w) \in Y^{+} \times Y^{+} \times W^{v}$, the set $\{\nu \in W^{v}\cdot\lambda \mid \mu \in R_{w}(\nu)\}$ is finite.
\end{lemme}
\begin{proof}
Let $N \geq 0$ be an integer that satisfies the following property:
\begin{equation}
\label{condition sur h et ell}
\forall \nu' \in W^{v}\cdot\lambda \mid \ell(w_{\nu'}) \geq N, \ h(\nu' - \lambda) < h(\mu - \lambda).
\end{equation}
(Such an integer exists by Lemma~\ref{lemComportement asymptotique hauteur longueur}.) Let $\nu \in W^{v}\cdot\lambda$ be such that $\mu$ belongs to $R_{w}(\nu)$ and set $u:=w_{\nu}$. Following Lemma~\ref{lemEcriture des elements de Rwlambda}, write $\mu = \sum_{x \in [1,w]} a_{x}x\cdot\nu$ with $(a_{x})_{x \in [1,w]} \in [0,1]^{[1,w]}$ such that $ \sum_{x \in [1,w]} a_{x} = 1$. For all $x \in [1,w]$, set $v(x):= w_{x\cdot\nu}$. Then there exists $x \in [1,w]$ such that $\ell(v(x)) < N$. Indeed, suppose by contradiction that $\ell(v(x)) \geq N$ for all $x \in [1,w]$. As
\[ \mu - \lambda = \sum_{x \in [1,w]} a_{x}(x\cdot\nu - \lambda) = \sum_{x \in [1,w]} a_{x}(v(x)\nu - \lambda),
\]
we obtain from \eqref{condition sur h et ell} that $$h(\mu - \lambda) = \sum_{x \in [1,w]} a_{x}h(v(x) - \lambda) < \sum_{x \in [1,w]} a_{x} h(\mu - \lambda) = h(\mu - \lambda),$$ which is absurd. We can hence pick $\overline{u} \in [1,w]$ such that $\ell(v(\overline{u})) < N$. As $\overline{u}\cdot\nu = v(\overline{u})\cdot\nu^{++}$, we have $\ell(\overline{u}^{-1}v(\overline{u})) \geq \ell(u)$ by definition of $u$. It implies that $\ell(v(\overline{u})) + \ell(u) \geq \ell(v(\overline{u})) + \ell(\overline{u}) \geq \ell(u)$, hence we have $\ell(u) \leq N + \ell(\overline{u})$, so $\ell(u)$ is upper-bounded and the lemma follows.
\end{proof}

To allow infinite sums in $\hat{\HH}$, we need a suitable notion of summable families, as we have by Definition \ref{def Familles sommables Looijenga algebra} for the Looijenga algebra $\R\llbracket Y\rrbracket$. This is the purpose of the next definition.
\begin{defn}
\label{def Familles sommables completed IH algebra}
A family $(a_j)_{j\in J}\in \hat{\HH}^J$ is \textbf{summable} when the following properties hold.
\begin{itemize}
\item For all $\lambda \in Y^{+}$, the set $\{j \in J \mid \exists w \in W^{v}, \, \pi_{\lambda, w}(a_{j}) \not= 0\}$ is finite.
\item $ \bigcup_{j \in J} \supp(a_{j})$ is almost finite.
\end{itemize}
If $(a_{j})_{j \in J} \in \hat{\HH}^{J}$ is a summable family, we define $ \sum_{j \in J} a_{j} \in \hat{\HH}$ by
\begin{multline*}
\sum_{j \in J} a_{j}:= \sum_{(\lambda, w) \in Y^{+} \times W^{v}} a_{\lambda, w}Z^{\lambda}H_{w},\\[-5pt]
 \text{with}\quad a_{\lambda,w}:= \sum_{j \in J} \pi_{\lambda,w}(a_{j}) \text{ for all } (\lambda, w) \in Y^{+} \times W^{v}.
\end{multline*}
\end{defn}
The next result claims that the product of two summable families is well-defined. This is the crucial step in the process that turns $\hat{\HH}$ into a convolution algebra for $*$. Recall that elements of $\HH$ corresponds to elements of $\hat{\HH}$ with finite support.
\begin{thm}
\label{Thm Convolution pour les familles sommables}
Let $(a_{j})_{j \in J} \in \HH^{J}$ and $(b_{k})_{k \in J} \in \HH^{K}$ be two summable families. Then $(a_{j} *b_{k})_{(j,k) \in J \times K}$ is summable and $ \sum_{(j,k) \in J \times K} a_{j} * b_{k}$ only depends on the two elements $ \sum_{j \in J} a_{j}$ and $ \sum_{k \in K} b_{k}$ of $\hat{\HH}$.
\end{thm}
\begin{proof} For any $j \in J$ and $k \in K$, we can decompose $a_{j}$ and $b_{k}$ as follows:
\[ a_{j} = \sum_{(\lambda, u) \in Y^{+} \times W^{v}} a_{j,\lambda, u}Z^{\lambda}H_{u} \quad \text{and} \quad b_{k} = \sum_{(\mu, v) \in Y^{+} \times W^{v}} b_{k,\mu,v}Z^{\mu}H_{v}.\]
For any $\lambda \in Y^{+}$, we set
\[ J(\lambda):= \{j \in J \mid \exists u \in W^{v}, \ a_{j,\lambda, u} \not= 0\} \ \text{and} \ K(\lambda):= \{k \in K \mid \exists v \in W^{v}, \ b_{k,\mu,v} \not= 0 \}.
\]
For any triple $(u,v,\mu) \in W^{v} \times W^{v} \times Y^{+}$, the application of Lemma~\ref{lem Reecriture relations commutation longueur qcq} to $H_{u}*\nobreak Z^{\mu}H_{v}$ gives a family $(z^{u,v,\mu}_{\nu, t})_{(\nu, t)\in R_{u}(\mu) \times [1,u]\cdot v}$ of scalars that satisfy $$ H_{u} * Z^{\mu}H_{v} = \sum_{(\nu, t) \in R_{u}(\mu) \times [1,u]\cdot v} z^{u,v,\mu}_{\nu,t}Z^{\nu}H_{t}.$$ Given $j \in J$ and $k \in K$, we then have
\begin{equation}
\label{egalite pour aj*bk}
a_{j} * b_{k} = \sum_{(\lambda,u), (\mu,v) \in Y^{+} \times W^{v}} \sum_{(\nu, t) \in R_{u}(\mu) \times [1,u]\cdot v} a_{j,\lambda, u} b_{k, \mu, v} z^{u,v,\mu}_{\nu,t} Z^{\lambda + \nu}H_{t}.
\end{equation}
This equality implies that $\supp_{W^{v}}(a_{j} * b_{k})$ is contained in $S^{a}_{j}.S^{b}_{k}$, where $S^{c}_{n}:= \bigcup_{w \in \supp_{W^{v}}(c_{n})} [1,w]$ for $c \in \{a,b\}$ and $n \in \{j,k\}$. This already gives the finiteness of
\[ S_{W^{v}}(a, b):= \bigcup_{(n,m) \in J \times K} \supp_{W^{v}}(a_{n} * b_{m}) \subset (\bigcup_{n\in J} S_{n}^a). (\bigcup_{m\in K}S_{m}^b).
\]
If we set $S:= \bigcup_{j \in J} \supp(a_{j}) \cup \bigcup_{k \in K} \supp(b_{k})$ and $S_{Y}\!:=\!\pi_{Y}(S)$, where \hbox{$\pi_{Y}{:}\, Y\!\times\!W^{v}\!\to\!Y$} is the projection on the first coordinate, then $S$ and $S_{Y}$ are by construction both almost finite. We can hence choose an integer $N \geq 0$ and elements $\kappa_{1}, \ldots, \kappa_{N} \in Y^{++}$ such that: for all $x \in S_{Y}$, there exists $i \in \llbracket 1, N \rrbracket$ such that $x^{++} \leq_{Q^{\vee}} \kappa_{i}$.

Now pick a pair $(\rho, s) \in Y^{+} \times W^{v}$. The image of $a_{j}*b_{k}$ by the projection $\pi_{\rho,s}$ is given by
\[ \pi_{\rho,s}(a_{j}*b_{k}) = \sum_{(\lambda,u), (\mu,v) \in Y^{+} \times W^{v}} \sum_{\nu \in R_{u}(\mu) \mid \lambda + \nu = \rho} a_{j,\lambda,u}b_{k,\mu,v}z^{u,v,\mu}_{\nu,s}.\]
Set $$F(\rho):= \{ (\lambda, \nu) \in S_{Y} \times Y^{+} \mid \exists (\mu, u) \in S_{Y} \times S_{W^{v}}, \nu \in R_{u}(\mu) \text{ and } \lambda + \nu = \rho \}.$$ If $(\lambda, \nu)$ is an element of $F(\rho)$, choose some $(\mu,u)\!\in\!S_{Y}\!\times\!S_{W^{v}}$ such that $\nu\!\in\!R_{u}(\mu)$ and $\lambda + \nu = \rho$. Then Lemma~\ref{lemPositivite de Rwlambda inferieur a lambda} implies the existence of $n,m \in \llbracket 1, N \rrbracket$ such that $\lambda \leq_{Q^{\vee}} \lambda^{++} \leq_{Q^{\vee}} \kappa_{n}$ and $\nu \leq_{Q^{\vee}} \mu^{++} \leq_{Q^{\vee}} \kappa_{m}$, which proves that $F(\rho)$ is finite.

Set $$F'(\rho):=\{\mu \in S_{Y} \mid \exists (u,(\lambda,\nu)) \in S_{W^{v}} \times F(\rho), \ \nu \in R_{u}(\mu) \}.$$ If $\mu$ is an element of $F'(\rho)$ and if $(u,(\lambda,\nu)) \in S_{W^{v}} \times F(\rho)$ is such that $\nu \in R_{u}(\mu)$, applying again Lemma~\ref{lemPositivite de Rwlambda inferieur a lambda} gives an integer $i \in \llbracket 1, N\rrbracket$ such that \hbox{$\nu^{++} \leq_{Q^{\vee}} \mu^{++} \leq_{Q^{\vee}} \kappa_{i}$}. This implies the finiteness of $F'(\rho)^{++}$, which implies itself the finiteness of $F'(\rho)$ by Lemma~\ref{lemFinitude des antecedants de mu par Rw}.

Set $$F_{1}(\rho):= \{ \lambda \in Y^{+} \mid \exists \nu \in Y^{+}, (\lambda, \nu) \in F(\rho)\text{ and }L(\rho):= \hspace*{-2mm}\bigcup_{(\lambda,\mu) \in F_{1}(\rho) \times F'(\rho)}\hspace*{-2mm} J(\lambda) \times K(\mu).$$ By construction, $L(\rho)$ is finite and for all $(j,k)\!\in\!J\!\times\!K$, the non-vanishing of \hbox{$\pi_{\rho,s}(a_{j} * b_{k})$} implies that $(j,k)$ belongs to $L(\rho)$. Also, if $(\rho,s)$ is in $$ \bigcup_{(j,k) \in J \times K} \supp(a_{j}*b_{k}),$$ then there exists $(\lambda, \mu) \in S_{Y} \times S_{Y}$, $u \in S_{W^{v}}$ and $\nu \in R_{u}(\mu)$ such that $\lambda + \nu = \rho$. Applying once more Lemma~\ref{lemPositivite de Rwlambda inferieur a lambda}, we get integers $n,m \in \llbracket 1,N\rrbracket$ such that $$\rho^{++} \leq_{Q^{\vee}} \lambda^{++} + \nu^{++} \leq_{Q^{\vee}} \kappa_{n} + \kappa_{m}.$$ Summed up, all this shows that $ \bigcup_{(j,k) \in J \times K} \supp(a_{j}*b_{k})$ is almost finite and that $(a_{j}*b_{k})_{(j,k) \in J \times K}$ is a summable family.

Moreover, we have
\[ \begin{aligned}
\pi_{\rho,s}\biggl(\sum_{(j,k) \in J \times K} a_{j}*b_{k} \biggr) &= \sum_{(\lambda,u), (\mu,v) \in Y^{+} \times W^{v}} \sum_{\nu \in R_{u}(\mu) \mid \lambda + \nu = \rho} \sum_{(j,k) \in J\times K} a_{j,\lambda, u} b_{k,\mu,v} z^{u,v,\mu}_{\nu,s}\\
&= \sum_{(\lambda,u), (\mu,v) \in Y^{+} \times W^{v}} \sum_{\nu \in R_{u}(\mu) \mid \lambda + \nu = \rho} a_{\lambda,u}b_{\mu,v}z^{u,v,\mu}_{v,s},
\end{aligned}
\]
where we set
\[
\sum_{j \in J} a_{j} = \sum_{(\lambda, u) \in Y^{+} \times W^{v}} a_{\lambda,u}Z^{\lambda}H_{u}\quad\text{and}\quad\sum_{k \in K} b_{k} = \sum_{(\mu, v) \in Y^{+} \times W^{v}} b_{\mu,v} Z^{\mu}H_{v},
\]
hence the lemma is proved.
\end{proof}

\begin{defn}
\label{DefinitionConvolution dans completed IH}
For any summable families $(a_{j})_{j \in J} \in \HH^{J}$ and $(b_{k})_{k \in J} \in \HH^{K}$, we set
\[ a*b:= \sum_{(j,k) \in J \times K} a_{j} * b_{k} \in \hat{\HH}, \quad \text{with } a:= \sum_{j \in J} a_{j} \text{ and } b:= \sum_{k \in K} b_{k}.\]
\end{defn}
\begin{cor}
\label{CorCompleted IH est une algebre}
The convolution product $*$ provides $\hat{\HH}$ with a structure of associative $\RR$-algebra.
\end{cor}
\begin{proof}
Theorem~\ref{Thm Convolution pour les familles sommables} ensures that $(\hat{\HH}, *)$ is an $\RR$-algebra. The associativity of $*$ in $\hat{\HH}$ comes from Theorem~\ref{Thm Convolution pour les familles sommables} and from the associativity of $*$ in $\HH$.
\end{proof}
The algebra $\hat{\HH}$ is called the \textbf{completed Iwahori-Hecke algebra of $(\A, (\sigma_{i}, \sigma'_{i})_{i \in I})$ over $\RR$}.

\begin{exmp}
Let $\I$ be a thick masure of finite thickness on which a group $G$ acts strongly transitively. For any $i \in I$, pick a panel $P_{i}$ of $\{x \in \A \mid \alpha_{i}(x) = 0 \}$ and a panel $P'_{i}$ of $\{x \in \A \mid \alpha_{i}(x) = 1 \}$. Let $1 + q_{i}$ (resp. $1 + q'_{i}$) be the number of chambers in $\I$ that contain $P_{i}$ (resp. $P'_{i}$) and set $\sigma_{i} = \sqrt{q_{i}}, \ \sigma'_{i} = \sqrt{q'_{i}}$. Then $(\sigma_{i}, \sigma'_{i})_{i \in I}$ satisfy the relations stated at the beginning of Section~\ref{secAlgebre d'IH completee} and the completed Iwahori-Hecke algebra of $(\A, (\sigma_{i}, \sigma'_{i})_{i \in I})$ over $\RR$ is called the \textbf{completed Iwahori-Hecke algebra of $\I$ over $\RR$}.
\end{exmp}

\subsection{Center of the Iwahori-Hecke algebras}
The goal of this section is to compute the center of the Iwahori-Hecke algebra $\HH$ and of its completed version $\hat{\HH}$. Our proof is basically an adaptation to this context of the proof of \cite[Th.\,1.4]{nelsen2003polynomials}. In the sequel, we denote by $\mathcal{Z}(A)$ the center of any $\RR$-algebra $A$.

\subsubsection{The completed Bernstein-Lusztig bimodule $\overline{\BLHH}$}
To determine $\mathcal{Z}(\hat{\HH})$, we want to compute elements of the form $Z^{\mu}*z*Z^{-\mu}$ for $z \in \mathcal{Z}(\hat{\HH})$ and $\mu \in Y^{+}$. However, left and right multiplication by $\Z^{\lambda}$ are only defined in $\hat{\HH}$ for $\lambda \in Y^{+}$. To~extend multiplication by $Z^{\lambda}$ for arbitrary $\lambda \in Y$, we need to pass to a bigger space: indeed, if $\lambda \in Y$ is not in $Y^{+}$, multiplication by $Z^{\lambda}$ obviously does not stabilize~$\hat{\HH}$, as $Z^{\lambda} * 1 = Z^{\lambda}$ is not in $\hat{\HH}$ in this case. The bigger space aforementioned is a ``completion'' $\overline{\BLHH}$ of $\overline{\BLHH}$ that contains $\hat{\HH}$. Note that $\overline{\BLHH}$ will not be equipped with a structure of algebra, but with a structure of $\RR[Y]$-bimodule compatible with the convolution product $*$ on $\hat{\HH}$.

Any $a=(a_{\lambda, w}) \in \RR^{Y \times W^{v}}$ will also be written as $a = \sum_{(\lambda, w) \in Y \times W^{v}} a_{\lambda,w} Z^{\lambda}H_{w}$. For such an $a$, we define the \textbf{support of $a$ along $W^{v}$} as $$\supp_{W^{v}}(a):= \{w \in W^{v} \mid \exists \lambda \in Y, a_{\lambda, w} \not= 0 \}.$$ We set $\overline{\BLHH}:= \{a \in \RR^{Y \times W^{v}} \mid \supp_{W^{v}}(a) \text{ is finite}\}$; note that $\BLHH$ and $\hat{\HH}$ can be seen as subspaces of $\overline{\BLHH}$. For any pair $(\rho, s) \in Y \times W^{v}$, we have again a projection map $\pi_{\rho, s}: \overline{\BLHH} \to \RR$ defined by:
\[ \forall \sum_{(\lambda,w)\in Y\times W^{v}} a_{\lambda,w} Z^{\lambda}H_{w}\in {\overline{\BLHH}}, \qquad \pi_{\rho,s}\biggl(\sum_{(\lambda, w) \in Y \times W^{v}} a_{\lambda,w} Z^{\lambda}H_{w}\biggr) = a_{\rho,s}.\]
\begin{defn}
\label{DefSommable dans Completed BL bimodule}
A family $(a_{j})_{j \in J} \in \overline{\BLHH}$ is \textbf{summable} if the following properties hold.
\begin{itemize}
\item For all pair $(\rho, s) \in Y \times W^{w}$, the set $\{j \in J \mid \pi_{s,\rho}(a_{j}) \not= 0\}$ is finite.
\item $ \bigcup_{j \in J} \supp_{W^{v}}(a_{j})$ is almost finite.
\end{itemize}
If $(a_{j})_{j \in J} \in \overline{\BLHH}$ is a summable family, we define $ \sum_{j\in J}a_{j} \in {\overline{\BLHH}}$ as
\[ \sum_{j\in J} a_{j}:=\sum_{(\lambda,w)\in Y\times W^{v}} a_{\lambda,w}Z^{\lambda} H_{w}, \quad\text{with } a_{\lambda,w}:= \sum_{j\in J} \pi_{\lambda,w}(a_{j}) \text{ for all }(\lambda,w)\in Y\times W^{v}.
\]
\end{defn}

\begin{lemme}
\label{lemDefinition action de Y sur completed BL bimodule}
Let $(a_{j})_{j \in J} \in (\BLHH)^{J}$ be a summable family in $\overline{\BLHH}$ and $a:= \sum_{j \in J} a_{j} \in \ \overline{\BLHH}$. For any $\mu \in Y$, $(a_{j} * Z^{\mu})_{j \in J}$ and $(Z^{\mu} * a_{j})_{j \in J}$ are summable families of $\overline{\BLHH}$, and the elements $ \sum_{j \in J} Z^{\mu}*a_{j}$ and $ \sum_{j \in J} a_{j} * Z^{\mu}$ only depend on $a$ and $\mu$ (but not on the choice of the family $(a_{j})_{j \in J}$.

Moreover, setting $a \ \overline{*} \ Z^{\mu}:= \sum_{j \in J} a_{j} * Z^{\mu}$ and $Z^{\mu} \ \overline{*} \ a:= \sum_{j \in J} Z^{\mu}*a_{j}$, we define a convolution product that provides $\overline{\BLHH}$ with a structure of $\RR[Y]$-bimodule.
\end{lemme}

\begin{proof}
Let $(a_{j})_{j \in J}\!\in\!(\BLHH)^{J}$ be a summable family and set $S\!:= \bigcup_{j \in J} \supp_{W^{v}}(a_{j})$. For all pair $(\lambda,w) \in Y \times W^{v}$, set
$J(\lambda, w):= \{j \in J \mid \pi_{\lambda, w}(a_{j}) \neq 0\}$. For all $(\mu, \rho, s, j) \in Y \times Y \times W^{v} \times J$, we have $ \pi_{\rho,s}\left(Z^{\mu} * a_{j}\right) = \pi_{\rho - \mu, s}(a_{j})$, hence the summability of $(Z^{\mu} * a_{j})_{j \in J}$ directly comes from the summability of $(a_{j})_{j \in J}$ and $$ \pi_{\rho,s}\biggl(\sum_{j \in J} Z^{\mu} * a_{j}\biggr) = \pi_{\rho-\mu,s}(a)$$ only depends on $a$ and $\mu$.

The corresponding statement for $(a_{j}*Z^{\mu})_{j \in J}$ is a little bit trickier to prove. Given $w \in W^{v}$, Lemma~\ref{lem Reecriture relations commutation longueur qcq} gives a family $(z^{w}_{\nu,t})_{(\nu, t) \in R_{w}(\mu) \times [1,w]}$ of coefficients in $\RR$ such that
\[ H_{w} * Z^{\mu} = \sum_{(\nu, t) \in R_{w}(\mu) \times [1,w]} z^{w}_{\nu,t} Z^{\nu}H_{t}.\]
For $j \in J$, write $a_{j} = \sum_{(\lambda, w) \in Y \times W^{v}} a_{j,\lambda,w}Z^{\lambda}H_{w}$ with $a_{j,\lambda,w} \in \RR$ for any pair $(\lambda,w) \in Y \times W^{v}$. Then we have, for all $(\mu,\rho,s,j) \in Y \times Y \times W^{v} \times J$:
\[\begin{aligned} \pi_{\rho,s}(a_{j}*Z^{\mu}) &= \pi_{\rho,s} \biggl(\sum_{(\lambda,w)\in Y\times S}a_{j,\lambda,w}Z^{\lambda} H_{w}*Z^{\mu} \biggr) \\ &= \pi_{\rho,s} \biggl(\sum_{(\lambda,w)\in Y\times S } \biggl(\sum_{(\nu, t) \in R_{w}(\mu)\times [1,w]}a_{j,\lambda,w} z^{w}_{\nu,t} Z^{\nu+\lambda} H_{t} \biggr) \biggr)\\ &= \sum_{(\lambda,w)\in Y\times S} \biggl(\sum_{\nu \in R_{w}(\mu) \mid \nu+\lambda=\rho} a_{j,\lambda,w}z^{w}_{\nu,s} \biggr).\end{aligned}\]
Fix $\mu \in Y$ and set $F(\rho,s):= \{j \in J \ \vert \pi_{\rho,s}(a_{j}*Z^{\mu}) \not= 0\}$ for all pair $(\rho,s) \in Y \times W^{v}$. By the previous computation, we have $ F(\rho, s) \subset \bigcup_{(w,\nu) \in S \times R_{w}(\mu)} J(\rho -\nu, w)$, hence $F(\rho, s)$ is finite. Moreover, for any $j \in J$, $\supp_{W^{v}}(a_{j}*Z^{\mu})$ is contained in $ \bigcup_{w \in S} [1,w]$, hence $ \bigcup_{j\in J} \supp_{W^{v}} (a_{j}*Z^{\mu})$ is almost finite and $(a_{j} * Z^{\mu})_{j \in J}$ is summable. Also note that the calculation of $ \pi_{\rho,s}\left(a_{j} * Z^{\mu}\right)$ we did above implies that for all $(\rho,s) \in Y \times W^{v}$, we have
\[ \begin{aligned} \pi_{\rho,s} \biggl(\sum_{j \in J} a_{j} * Z^{\mu} \biggr) & = \sum_{j \in J} \biggl(\sum_{(\lambda, w) \in Y \times S} \biggl(\sum_{\nu \in R_{w}(\mu) \mid \nu + \lambda = \rho} a_{j,\lambda,w} z^{v}_{\nu,s} \biggr) \biggr)\\
&= \sum_{(\lambda, w) \in Y \times S} \biggl(\sum_{\nu \in R_{w}(\mu) \mid \nu + \lambda = \rho} \biggl(\sum_{j \in J} a_{j,\lambda, w} z^{w}_{\nu, s} \biggr) \biggr),
\end{aligned}
\]
hence if $a = \sum_{(\lambda, w) \in Y \times W^{v}} a_{\lambda,w} Z^{\lambda}H_{w}$, then $$\pi_{\rho,s} \biggl(\sum_{j \in J} a_{j} * Z^{\mu} \biggr) = \sum_{(\lambda,w) \in Y \times S} \sum_{\nu \in R_{w}(\mu) \mid \nu + \lambda = \rho} a_{\lambda,w} z^{w}_{\nu,s}$$ only depends on $a$ and $\mu$.

To conclude the proof, we are left to show that for any $(b, \mu, \mu') \in Y^{3}$, we have
\begin{align*}
Z^{\mu}\ \overline{*}\ (Z^{\mu'}\ \overline{*}\ b)&=(Z^{\mu+\mu'})\ \overline{*}\ b, \ (b\ \overline{*}\ Z^\mu)\ \overline{*}\ Z^{\mu'}=b\ \overline{*}\ (Z^{\mu+\mu'})\\
\tag*{and} Z^\mu\ \overline{*}\ (b\ \overline{*}\ Z^{\mu'})&=(Z^\mu\ \overline{*}\ b)\ \overline{*}\ Z^{\mu'}.
\end{align*}
To do this, write $b = \sum_{(\lambda, w) \in Y \times W^{v}} b_{\lambda,w} Z^{\lambda}H_{w}$ with $(b_{\lambda,w}) \in \RR^{Y \times W^{v}}$ and apply the first part of this lemma to $J = Y \times W^{v}$: by associativity of $*$ in $\BLHH$, we get the required identities.
\end{proof}

\begin{cor}
\label{corCompatibilite convolutions}
For all $a\in \hat{\HH}$ and $\mu\in Y^{+}$, we have $Z^{\mu} * a=Z^{\mu}\ \overline{*}\ a$ and $a*Z^{\mu}=a\ \overline{*}\ Z^{\mu}$.
\end{cor}
This statement justifies that we will from now on denote $*$ instead of $\overline{*}$.

\subsubsection{Computation of the centers}
\label{subsubsectionCalcul des centres}
\begin{lemme}
\label{lemCommutativite action Y sur le centre de IH completee}
For all $a \in \mathcal{Z}(\hat{\HH})$ and $\mu \in Y$, we have $a*Z^{\mu} = Z^{\mu}*a$.
\end{lemme}
\begin{proof}
Write $\mu = \mu_{+} - \mu_{-}$ with $\mu_{+}, \mu_{-} \in Y^{+}$. The associativity of $*$ proven in Lemma \ref{lemDefinition action de Y sur completed BL bimodule} implies that $Z^{\mu_{-}} *(Z^{-\mu_{-}}*a) = a$, hence $Z^{-\mu_{-}}*a = a * Z^{-\mu_{-}}$ and $Z^{\mu}*a = Z^{\mu_{+}} * a * Z^{-\mu_{-}} = a * Z^{\mu}$.
\end{proof}

For any $w \in W^{v}$, we introduce the following subsets of $\overline{\BLHH}$:
\[ \left\{ \begin{aligned}
\overline{\BLHH}_{\not\geq w} &:= \left\{ \textstyle\sum_{(\lambda, v) \in Y \times W^{v}} a_{\lambda,v}Z^{\lambda}H_{v} \in \ \overline{\BLHH} \mid (a_{\lambda,v} \not= 0) \Rightarrow (v \not\geq w)\right\}; \\
\overline{\BLHH}_{=w} &:= \left\{ \textstyle\sum_{(\lambda, v) \in Y \times W^{v}} a_{\lambda,v}Z^{\lambda}H_{v} \in \ \overline{\BLHH} \mid (v \not= w) \Rightarrow (a_{\lambda,v} = 0) \right\}. \\
\end{aligned}\right.\]
We let $\hat{\HH}_{\not\geq w}:= \overline{\BLHH}_{\not\geq w} \cap \hat{\HH}$ and $\hat{\HH}_{=w}:= \overline{\BLHH}_{=w} \cap \hat{\HH}$ be the corresponding subspaces in $\hat{\HH}$.

\begin{lemme}
\label{lemStabilite des Hw}
Let $w \in W^{v}$ and $\lambda \in Y$.
\begin{enumerate}
\item We have $$\overline{\BLHH}_{\ngeq w} * Z^{\lambda} \subset {\overline{\BLHH}}_{\ngeq w},\quad Z^{\lambda}*{\overline{\BLHH}}_{\ngeq w} \subset {\overline{\BLHH}}_{\ngeq w}\quad\text{and}\quad Z^{\lambda} * {\overline{\BLHH}}_{=w} \subset {\overline{\BLHH}}_{= w}.$$
\item There exists $S \in {\overline{\BLHH}}_{\ngeq w}$ such that $H_{w}*Z^{\lambda} = Z^{w(\lambda)}H_{w} + S$.
\end{enumerate}
\end{lemme}
\begin{proof}
These statements are consequence of \cite[Th.\,6.2]{bardy2016iwahori}, of Lemma~\ref{lem Reecriture relations commutation longueur qcq} and of
Lemma~\ref{lemDefinition action de Y sur completed BL bimodule}.
\end{proof}

The following theorem is the heart of this section, as it describes the center of the completed Iwahori-Hecke algebra $\hat{\HH}$. This generalizes a well-known theorem of Bernstein (see \cite[Th.\,8.1]{lusztig1983singularities}, which seems to be the first published version of this result) and gives a recovery of the spherical Hecke algebra $\HH_{s}$ as center of a natural Iwahori-Hecke algebra.

\begin{thm}
\label{ThmCentreIHcompletee}
The center of the completed Iwahori-Hecke algebra $\hat{\HH}$ is $\mathcal{Z}(\hat{\HH}) = R\llbracket Y\rrbracket^{W^{v}}$.
\end{thm}
\begin{proof}
Let $a = \sum_{\lambda \in Y^{+}} a_{\lambda}Z^{\lambda}$ be an element of $R\llbracket Y\rrbracket^{W^{v}}$ and $i \in I$. We can write $a = x + y$ with $x = \sum_{\lambda \in Y^{+} \cap \ker \alpha_{i}} a_{\lambda}Z^{\lambda}$ and $y = \sum_{\lambda \in Y^{+} \mid \alpha_{i}(\lambda) >0} a_{\lambda} (Z^{\lambda} + Z^{r_{i}(\lambda)})$. As~$x$ and $y$ commute with $H_{i}$, we obtain that $a$ commutes with $H_{i}$ for all $i \in I$, hence we have $a \in \mathcal{Z}(\hat{\HH})$ and $R\llbracket Y\rrbracket^{W^{v}} \subset \mathcal{Z}(\hat{\HH})$.

Conversely, let $z$ be an element of $\mathcal{Z}(\hat{\HH}) \subset {\overline{\BLHH}}$ and write $$z = \sum_{(\lambda, w) \in Y \times W^{v}} c_{\lambda,w}Z^{\lambda}H_{w}.$$ First assume that the set $$F = \{(\lambda,w) \in Y \times W^{v} \mid w \not= 1 \text{ and } c_{\lambda,w} \not= 0\}$$ is non empty and choose a pair $(\nu,m) \in F$ with $m$ maximal in $W^{v}$ (for the Bruhat order). Write $z = x + y$ with $x = \sum_{\lambda \in Y} x_{\lambda, m} Z^{\lambda}H_{m} \in \hat{\HH}_{=m}$ and $y\in \hat{\HH}_{\ngeq m}$. Lemmas~\ref{lemCommutativite action Y sur le centre de IH completee} and~\ref{lemStabilite des Hw} imply that for all $y \in Y$, we have $$ z = Z^{\mu} * z * Z^{-\mu} = \sum_{\lambda \in Y} c_{\lambda,m}Z^{\lambda+\mu - m(\mu)}H_{m} + y_{1}$$ with $y_{1} \in {\overline{\BLHH}}_{\ngeq m}$. By projection on ${\overline{\BLHH}}_{=m}$, we get that
$$x = \sum_{\lambda \in Y} c_{\lambda, m} Z^{\lambda + \mu - m(\mu)} H_{m}.$$ Now let $J \subset Y$ be a finite set that satisfies the following property: $$ \forall (\lambda, w) \in Y \times W^{v}, \ (c_{\lambda, w} \not= 0) \Longrightarrow (\exists \nu \in J \mid \lambda \leq_{Q^{\vee}} \nu).$$ Pick $\gamma \in Y$ such that $c_{\gamma, m} \not= 0$.
Then for all $\mu \in Y$, we have $c_{\gamma + \mu - m(\mu), m} \not=0$, hence there exists $\nu(\mu) \in J$ such that $\gamma + \mu - m(\mu) \leq_{Q^{\vee}} \nu(\mu)$. In particular, pick $\mu \in Y \cap C^{v}_{f}$ and let $\nu \in J$ be such that for all integer $n \geq 0$, we have $\gamma + \sigma(n)(\mu - m(\mu)) \leq_{Q^{\vee}} \nu$, where $\sigma: \Z_{+} \to \Z_{+}$ is such that $ \lim_{n \to +\infty} \sigma(n) = +\infty$.
Then $\gamma + \sigma(1)(\mu - m(\mu)) - \nu$ belongs to $Q^{\vee}$, and Lemma~\ref{lemLemme 2.4 a) de GR14} implies that $\mu - m(\mu)$ is a non-zero element of $Q^{\vee}_{+}$. Hence for $n$ large enough, we have $$\gamma+ \sigma(n)(\mu-m(\mu))= \gamma+\sigma(1)(\mu-m(\mu))+ (\sigma(n)-\sigma(1))(\mu-m(\mu)) >_{Q^\vee} \nu,$$ which contradicts the definition of $\nu$.
Consequently, $F$ is empty and $z$ belongs to $R\llbracket Y\rrbracket$. We can hence simplify the above decomposition of $z$ and write $z = \sum_{\lambda \in Y} c_{\lambda}Z^{\lambda}$ with $c_{\lambda} = c_{\lambda, 1}$. By Lemma~\ref{lemStabilite des Hw}, we know that for any $w \in W^{v}$, we have $$H_{w} * z = \sum_{\lambda \in Y} Z^{w(\lambda)}H_{w} + y$$ for some $y \in {\overline{\BLHH}}_{\ngeq w}$. But $z$ commutes with $H_{w}$, so we also have $$ H_{w} * z = z * H_{w} = \sum_{\lambda \in Y} c_{\lambda}Z^{\lambda}H_{w}.$$ By projection on $\hat{\HH}_{=w}$, we get that $$ \sum_{\lambda \in Y} c_{\lambda}Z^{w(\lambda)}H_{w} = \sum_{\lambda \in Y} c_{\lambda} Z^{\lambda}H_{w},$$ hence $z$ is in $R\llbracket Y\rrbracket^{W^{v}}$, which ends the proof.
\end{proof}

As a consequence of Theorem \ref{ThmCentreIHcompletee}, we get a description of the center of the usual Iwahori-Hecke algebra $\HH$. Note that the proof relies on a characterization of finite $W^{v}$-orbits in $\A$ that will be proven independently at the end of Section \ref{secAlgebres de Hecke generales} (see Corollary~\ref{corOrbites infinies}).

Before we state the result, let us recall some notations. If $A_{1}, \ldots, A_{r}$ denote the indecomposable components of the Kac-Moody matrix $A$, we let $J^{f}$ be the set of all $j \in \llbracket 1,r \rrbracket$ such that $A_{j}$ is of finite type \cite[Th.\,4.3]{kac1994infinite} and $J^{\infty}$ be the complement of $J^{f}$ in $\llbracket 1, r\rrbracket$. Set $\A^{f}:= \bigoplus_{j\in J^{f}} \A_{j}$, let $\Phi_{j}$ be the root system of $\A_{j}$ and \hbox{$\A_{j,in}:= \bigcap_{\phi\in \Phi_j}\ker \phi$ for all $j \in J^{f}$}. Finally, set $Y^{f}:= Y \cap \A^{f}$, $\A^{\infty}_\mathrm{in}:= \bigoplus_{j\in J^\infty} \A_{j,in}$ and $Y^{\infty}_\mathrm{in}:= Y \cap A^{\infty}_\mathrm{in}$.

\begin{lemme}
\label{lemCentre IH usuel}
We have $\mathcal{Z}(\HH) = \mathcal{Z}(\hat{\HH}) \cap \HH = R[Y^{f}\oplus Y^{\infty}_\mathrm{in}]$.
\end{lemme}
\begin{proof}
Any $a \in \mathcal{Z}(\HH)$ is in $\HH$ and satisfies $a*Z^{\lambda}H_{w} = Z^{\lambda}H_{w}*a$ for all $(\lambda,w) \in Y \times W^{v}$, hence belongs to $\mathcal{Z}(\hat{\HH})$ by Theorem~\ref{Thm Convolution pour les familles sommables}. As the other inclusion is clear, we already have that $\mathcal{Z}(\HH) = \mathcal{Z}(\hat{\HH}) \cap \HH$. By Theorem~\ref{ThmCentreIHcompletee}, we get that $\mathcal{Z}(\HH) = \HH \cap R\llbracket Y\rrbracket^{W^v}$, and Corollary~\ref{corOrbites infinies} now implies that this intersection is reduced to $ R[Y^{f}\oplus Y^{\infty}_\mathrm{in}]$, which ends the proof.
\end{proof}

\begin{rque}
\label{RemModule sur le centre de HH completee}
When $W^{v}$ is finite, it is well-known that $\HH$ is a finitely generated $\mathcal{Z}(\HH)$-module, and it is natural to wonder whether the corresponding statement holds in the infinite case. Unfortunately, when $W^{v}$ is infinite, $\hat{\HH}$ is not of finite type over $\mathcal{Z}(\hat{\HH})$. Indeed, let $J$ be any finite set and pick any finite family $(h_{j})_{j \in J} \in \hat{\HH}^{J}$. For all $(z_{j})_{j \in J} \in \mathcal{Z}(\hat{\HH})^{J}$, we have $$ \supp _{W^v} \biggl(\sum_{j\in J} z_{j} h_{j}\biggr)\subset \bigcup_{j\in J} \supp_{W^{v}} (h_{j})\subsetneq W^{v},$$ hence $(h_{j})_{j \in J}$ cannot span $\hat{\HH}$ over $\mathcal{Z}(\hat{\HH})$.
\end{rque}

\subsection{Some further remarks}
\subsubsection{The special case of reductive groups}
Assume in this paragraph that $G$ is reductive, in which case $\T = \AAA$ and $Y = Y^{+}$. Then almost finite sets as defined in \cite{gaussent2014spherical} are finite sets: indeed, the Kac-Moody matrix $A$ is in this case a Cartan matrix, hence it satisfies condition (FIN) in \cite[Th.\,4.3]{kac1994infinite}. In particular, $Y^{++}$ is contained in $Q^{\vee}_{+} \oplus \A_\mathrm{in}$, so to be a subset of some $ \bigl(\bigcup_{i = 1}^{k} (y_{i} - Q^{\vee}_{+})\bigr) \cap Y^{++}$ amounts to be finite. Though the algebra $\hat{\HH}$ is still different from $\HH$, as $ \sum_{\mu \in Q^{\vee}_{+}} Z^{-\mu}$ is for instance an element of $\hat{\HH}$ that is not in $\HH$, they both have the same center. Indeed, we have the following result.
\begin{prop}
\label{propCNS egalite elements W^v invariants presque finis ou finis}
Let $R$ be any ring. Then $R\llbracket Y\rrbracket^{W^{v}}=R[Y]^{W^{v}}$ if and only if $W^{v}$ is finite.
\end{prop}
\begin{proof}
If $W^{v}$ is infinite, then for any $y \in Y \cap C^{v}_{f}$, the element $ \sum_{w \in W^{v}} e^{w\cdot y}$ belongs to $R\llbracket Y\rrbracket^{W^{v}}$ but not to $R[Y]^{W^{v}}$, hence $R\llbracket Y\rrbracket^{W^{v}}\not=R[Y]^{W^{v}}$.

If $W^{v}$ is finite, let $w_{0}$ be the longest element of $W^{v}$. By \cite[\S 1.8]{humphreys1992reflection}, we know that $w_{0}.Q^{\vee}_{+} = Q^{\vee}_{-}$. If $E \subset Y$ is almost finite, there is some finite set $J$ and a family $(y_{j})_{j \in J} \in Y^{J}$ such that $E$ is contained in $ \bigcup_{j \in J} (y_{j} - Q^{\vee}_{+})$. If E is furthermore $W^{v}$-invariant, then $E = w_{0}\cdot E$ is also contained in $ \bigcup_{j\in J} (w_{0}\cdot y_{j}+Q^{\vee}_{+})$, hence any element $x \in E$ satisfies $w_{0}\cdot y_{j} \leq_{Q^{\vee}} x \leq_{Q^{\vee}} y_{j'}$ for some $j,j' \in J$. This implies that $E$ is finite and completes the proof.
\end{proof}
Using \cite[Th.\,8.1]{lusztig1983singularities}, Theorem~\ref{ThmCentreIHcompletee} and Lemma~\ref{lemCentre IH usuel}, we get from Proposition~\ref{propCNS egalite elements W^v invariants presque finis ou finis} that when $W^{v}$ is finite, we have:
\[\mathcal{Z}(\hat{\HH})=R[Y]^{W^{v}}=\mathcal{Z}(\HH).\]

\subsubsection{Iwahori-Hecke algebras and $K_{I}$ double cosets}
In the completion process we used to define $\hat{\HH}$, we used the Bernstein-Lusztig relations of $\HH$. However, the Iwahori-Hecke algebra $\HH$ is initially defined in a different way, namely as a convolution algebra of $K_{I}$-bi-invariant functions. In particular, a natural basis of $\HH$ is given by characteristic functions of $K_{I}$ double cosets, and the Bernstein-Lusztig presentation comes afterwards. This leads naturally to address the following question: can we see the completed algebra at the level of $K_{I}$ double cosets, as it is the case for the spherical Hecke algebra?

Fix a ring $\RR$ as before and let $\mathcal{C}_{0}$ be the set of positive type $0$ chambers. Set $W^{+}:= W^{v} \ltimes Y^{+}$ and let $d^{W}$ be the distance defined in \cite{bardy2016iwahori} (see also Section~\ref{subsecDistances associees aux faces de type $0$} below). Recall that $\HH$ is isomorphic to $ \bigoplus_{\wt \in W^{+}} \RR T_{\wt}$ for the product defined by $T_{\wt} * T_{\vt} = \sum_{\ut \in W^{+}} a^{\ut}_{\wt, \vt}$ for all elements $\wt, \vt \in W^{+}$, provided that we set, for all $\ut \in W^{+}$,
\[a^{\ut}_{\wt, \vt}:= \bigl|\{C\in \mathcal{C}_{0} \mid C_{0}^{+}\leq C\leq \ut\cdot C_{0}^{+},\ d^{W}(C_{0}^{+},C)=\wt\mathrm{\ and\ }d^{W}(C,\ut\cdot C_{0}^{+})=\vt\}\bigr|.\]
For $x = (x_{\wt})_{\wt \in W^{+}} \in \RR^{W^{+}}$, also write $x = \sum_{\wt \in W^{+}} x_{\wt} T_{\wt}$. For now, we do not know whether it is possible to endow $\RR^{W^{+}}$, or some subspace $\overline{\HH} \subset \RR^{W^{+}}$ containing $\HH$, with a product that extends the convolution product of $\HH$. At least in general, it seems difficult to embed $\widehat{\HH}$ into $\RR^{W^{+}}$. Indeed, assume for instance that $\RR = \C$ and let $\pi: \C^{W^{+}} \to \C$ be the map defined by
\[ \pi\biggl(\sum_{\wt \in W^{+}} x_{\wt}T_{\wt}\biggr):= x_{1} \ \text{ for all } x = (x_{\wt})_{\wt \in W^{+}} \in \C^{W^{+}}.
\]
(Here, $1$ denotes the identity element in $W^{+}$.) When $G$ is reductive, we know for instance by \cite[Cor.\,1.9]{opdam2003generating} that for any $\lambda \in -Q^{\vee}$, $\pi(Z^{\lambda})$ is a positive real number, which makes it apparently hard to consider $ \sum_{\lambda\in -Q^\vee} \frac{1}{\pi(Z^\lambda)}Z^{\lambda} \in \widehat{\HH}$ as an element of $\C^{W^{+}}$. In the non-reductive case, we do not know so far whether an analogue of \cite[Cor.\,1.9]{opdam2003generating} is true.

\section{Hecke algebra associated with a parahoric subgroup}
\label{secAlgebres de Hecke generales}
The goal of this section is to attach a Hecke algebra to other subgroups than~$K_{s}$ or~$K_{I}$, by generalizing previous constructions of \cite{braverman2016iwahori}
and \cite{bardy2016iwahori} for the Iwahori subgroup $K_{I}$. Our motivation comes from the reductive case, where Hecke algebras can be associated with any open compact subgroup (see Section \ref{subsecMotivationHeckeFaceType0} below). When $G$ is not reductive, we know from Theorem \ref{Thm Pas de topologie ouverte compacte pour les fixateurs} that there is no reasonable topology on $G$, hence we cannot define ``open compact'' in our context. Nevertheless, there is still a notion of \textbf{special parahoric subgroup}, defined as the fixer of a type $0$ face of the masure $\I$.

Given a special parahoric subgroup $K = K_{F}$ that fixes a spherical type $0$ face $F$ satisfying $F_{0} \subset \overline{F} \subset \overline{C^{+}_{0}}$, we will generalize the construction done for $K_{I}$ by Bardy-Panse, Gaussent and Rousseau \cite{gaussent2014spherical, bardy2016iwahori} to build a Hecke algebra associated with $K_{F}$. This requires some finiteness results that fails anytime $F \not= F_{0}$ is not spherical (see Section \ref{subsecCasnonspherique}).

\subsection{Motivation from the reductive case}
\label{subsecMotivationHeckeFaceType0}
To motivate our definition in the Kac-Moody case (see Definition \ref{DefAlgebre de Hecke pour face spherique type $0$}), we start by recalling the classical setting for reductive groups. This section follows \cite[I.3.3]{vigneras1996representations}, though the idea of considering Hecke algebras as spaces of bi-invariant functions goes back at least to Weil and Shimura \cite{shimura1959integrales}, and to Iwahori \cite{iwahori1964hecke} and Iwahori-Matsumoto \cite{iwahori1965bruhat}.

Assume that $G$ is reductive, in which case it is naturally endowed with a structure of topological group induced by the topology of $\mathcal{K}$. For any open compact subgroup~$K$ of~$G$, let $\Z_{c}(G/K)$ be the space of compactly supported functions $G \to \Z$ that are $K$-invariant under right multiplication. Define an action of $G$ on this space by setting $g\cdot f:=[ x \mapsto f(g\cdot x)]$ for all pairs $(g,f) \in G \times \Z_{c}(G/K)$. The algebra $H(G,K):= \mathrm{End}_{G}(\Z_{c}(G/K))$ of $G$-equivariant endomorphisms of $\Z_{c}(G/K)$ is called the \textbf{Hecke algebra of $G$ relative to $K$}. If $\Z_{c}(G//K)$ is the ring of compactly supported functions $G \to \Z$ that are $K$-bi-invariant (for left and right multiplications), then we have a natural isomorphism of algebras $\Upsilon: H(G,K) \to \Z_{c}(G//K)$ given by $\Upsilon(\phi):= \phi(\mathds{1}_{K})$ for all $\phi \in H(G,K)$. This shows that $H(G,K)$ is a free $\Z$-algebra with canonical basis $\{e_{g} = \mathds{1}_{KgK}, \ g \in K \backslash G / K \}$. The product of two elements $e_{g}, e_{g'}$ of this basis is given~by
\begin{equation}
\label{FormuleProduitConvolutionAlgebreHecke}
e_{g}.e_{g'} = \hspace*{-2mm}\sum_{g'' \in K \backslash G / K}\hspace*{-2mm} m(g,g'; g'') e_{g''} \ \text{ with } m(g,g';g''):=|(KgK\cap g''Kg'^{-1} K)/K|.
\end{equation}
(Note that the non-vanishing of $m(g,g';g'')$ implies that $Kg''K$ is contained in $KgKg'K$.)

Extension of scalars works as follows: for any commutative ring $R$, the algebra $H_{R}(G,K):= H(G,K) \otimes_{\Z} R$ is called the \textbf{Hecke algebra of $G$ over $R$ relative to $K$}.

When $G$ is not reductive, we will replace open compact subgroups (that are not defined) by special parahoric subgroups. More precisely, let $K = K_{F}$ be the fixer in $G$ of a type $0$ face $F$ that satisfies $F_{0} \subset \overline{F} \subset \overline{C}^{+}_{0}$. Following what is done in \cite{bardy2016iwahori} in the Iwahori case, we will see $(KgK\cap g''Kg'^{-1} K)/K$ as intersection of ``spheres'' in~$\I$ and prove that this intersection is finite when $F$ is spherical (see Lemma~\ref{lemFinitude des parametres de l'algebre}) but infinite when $F \not= F_{0}$ is not spherical (see Proposition~\ref{propNon definition du produit pour les non-spheriques}). Hence for $F$ spherical, we will be able to define the Hecke algebra $\FHH$ associated with $K_{F}$ as the free $\Z$\nobreakdash-module with 
basis $\{e_{g} = \mathds{1}_{KgK}, \ g \in K \backslash G^{+} / K \}$, where $G^{+}:= \{g \in G \mid g\cdot0 \geq 0\}$, equipped with the convolution product given by the analogue of formula \eqref{FormuleProduitConvolutionAlgebreHecke} with $g'' \in G^{+}$. To prove this, we use the fact that these results are already known when $F$ is a type $0$ chamber (by \cite{bardy2016iwahori}), and the finiteness of the number of type $0$ chambers dominating $F$ as above. Note that the use of $G^{+}$ instead of $G$ in the definition of $\FHH$ is related to the fact that two points of $\I$ do not always lie in a same apartment. This change of group already shows up in the spherical and the Iwahori cases (see \cite{bardy2016iwahori, braverman2011spherical, braverman2016iwahori, gaussent2014spherical}).

From now on, we fix a type $0$ face $F$ that satisfy $F_{0} \subset \overline{F} \subset \overline{C^{+}_{0}}$. We denote by $K = K_{F}$ its fixer in $G$ and by $W_{F}$ its pointwise fixer in $W^{v}$. Then $F$ is \textbf{spherical} when~$W_{F}$ is finite. We also let $\Delta_{F} = G\cdot F$ be its orbit under the action of $G$ on $\I$. Note that we have a bijection $\Upsilon_{F}: G/K \to \Delta_{F}$ that maps $g\cdot K_{F}$ to $g\cdot F$.

\subsection{Distance and spheres associated with a type $0$ face}
\label{subsecDistances associees aux faces de type $0$}
In this section, we define an ``$F$-distance'' (or ``$W_{F}$-distance'') that generalizes the $W^{v}$\nobreakdash-distance introduced in \cite{gaussent2014spherical} and the $W$-distance defined in \cite{bardy2016iwahori}.

If $A$ (resp. $A'$) is an apartment of $\I$ and if $E_{1}, \ldots, E_{k}$ (resp. $E'_{1}, \ldots, E'_{k}$) are subsets or filters of $A$ (resp. $A'$), we denote by $\phi: (A,E_{1}, \ldots, E_{k}) \to (A', E'_{1}, \ldots, E'_{k})$ any isomorphism of apartment $\phi: A \to A'$ induced by some element of $G$ and such that: $\forall i \in \llbracket 1,k\rrbracket$, $\phi(E_{i}) = E'_{i}$.
When we do not want to precise which apartments $A$ and $A'$ are chosen, we simply write $\phi:(E_{1},\ldots,E_{k})\to (E'_{1},\ldots, E'_{k})$.

We define a relation $\leq$ on $\Delta_{F}$ as follows: for $F_{1}, F_{2}$ in $\Delta_{F}$, we write $F_{1} \leq F_{2}$ when $a_{1} \leq a_{2}$, where $a_{i}$ denotes the vertex of $F_{i}$ for $i \in \{1,2\}$. We then set
\[ \Delta_{F}\times_{\leq} \Delta_{F}:= \{(F_{1},F_{2})\in \Delta_{F}^{2} \mid F_{1} \leq F_{2} \}.\]
For any $F' \in \Delta_{F} \cap \A$, we set $[F']:= W_{F}\cdot F'$.

\begin{prop}
\label{propDefinition de la F-distance}
For all $(F_{1}, F_{2}) \in \Delta_{F} \times_{\leq} \Delta_{F}$, there exists an apartment $A$ containing $F_{1}$ and $F_{2}$, and an isomorphism $\phi: (A, F_{1}) \to (\A, F)$. Moreover, $d^{F}(F_{1}, F_{2}):= [\phi(F_{2})]$ only depends on the pair $(F_{1},F_{2})$.
\end{prop}
\begin{proof}
Given $(F_{1},F_{2}) \in \Delta_{F} \times_{\leq} \Delta_{F}$, the existence of an apartment $A$ containing~$F_{1}$ and~$F_{2}$ comes from \cite[Prop.\,5.1]{rousseau2011masures}. By construction, there is some $g \in G$ such that $F_{1} = g\cdot F$. Let $A':= g\cdot\A$: by (MA2), there exists an isomorphism \hbox{$\psi: (A,F_{1}) \to (A', F_{1})$}. Set $\psi':= g_{|\A}^{|A'}$: then $\phi:= \psi'^{-1} \circ \psi$ has the required properties.

Let now $A_{1}$ be another apartment containing $F_{1}$ and $F_{2}$ and $\phi_{1}: (A_{1}, F_{1}) \to (\A, F)$ be another suitable isomorphism. By \cite[Th.\,5.18]{hebert2017convexity}, there exists an isomorphism $f: (A, F_{1}, F_{2}) \to (A_{1}, F_{1}, F_{2})$. We hence have the following commutative diagram:
\[\xymatrix{ (\A,F_1,F_2)\ar[d]^{\phi}\ar[r]^{f} & (A_1,F_1,F_2)\ar[d]^{\phi_1}\\ (\A,F,\phi(F_2))\ar[r] & (\A,F, \phi_1(F_2)),}
\]
with the lower horizontal arrow that is induced by an element of $W_{F}$, hence $[\phi(F_{2})] = [\phi_{1}(F_{2})]$ does not depend on any choice and the proof is complete.
\end{proof}

\begin{rque}
\label{rqueModification preuve distance dans cas particuliers}
Proposition \ref{propDefinition de la F-distance} does not require $F$ to be spherical, thought it is the most important case for us. When $F$ is spherical, one can use \cite[Prop.\,1.10c)]{bardy2016iwahori} instead of \cite[Th.\,5.18]{hebert2017convexity}. Note that in the sequel, we will only use the $F$-distance attached to a non-spherical face for pairs of type $0$ faces based at the same vertex. In this special case, \cite[Th.\,5.18]{hebert2017convexity} could be replace by \cite[Prop.\,5.2]{rousseau2011masures}.
\end{rque}

\begin{rque}
\label{rqueComparaison aux autres distances vectorielles}
When $F = F_{0}$, we can identify $d^{F}$ with the ``vectorial distance'' $d^{v}$ of \cite{gaussent2014spherical} through the usual bijections $\Delta_{F_{0}} \simeq G\cdot0$ and $Y^{++} \simeq Y^{+}/W^{v}$.

When $F = C^{+}_{0}$, we have $W_{C^{+}_{0}} = \{1\}$, hence $[C] = \{C\}$ for all chamber $C \in \Delta_{C^{+}_{0}}$ and $d^{C^{+}_{0}}$ can be identified with the distance $d^{W}$ of \cite{bardy2016iwahori}, provided that each element $w$ of $W^{v} \ltimes Y^{+}$ is identified with the type $0$ chamber $w\cdot C^{+}_{0}$.
\end{rque}

Set $$\Delta_{\geq F}^{\A}:= \left\{ E \in \Delta_{F} \cap \A \mid E \geq F\right\}\quad\text{and}\quad[\Delta_{F}]:=\{ [F'], F' \in \Delta_{\geq F}^{\A}\}.$$ Moreover, for any pair $(E, [R]) \in \Delta_{F} \times [\Delta_{F}]$, set
\begin{align*}
\mathcal{S}^{F}(E,[R])&:= \{E' \in \Delta_{F} \mid E' \geq E \ \text{ and } d^{F}(E,E') = [R]\},\\
\mathcal{S}_{\opp}^{F}(E,[R])&:= \{E' \in \Delta_{F} \mid E' \leq E,\text{ and } d^{F}(E,E') = [R]\}.
\end{align*}

For any $E \in \Delta_{\geq F}^{\A}$, we choose some $g_{E} \in N$ such that $E = g_{E}\cdot F$. Such an element exists: indeed, let $g \in G$ be such that $E = g\cdot F$ and set $A:= g\cdot\A$. By (MA2) and \cite[2.2.1)]{rousseau2011masures}, we get an isomorphism $\phi: (A, g\cdot F) \to (\A, g\cdot F)$. Letting $\psi=g_{|\A}^{|A}$, we have $\phi \circ \psi \in N$ such that $(\phi \circ \psi)(F) = \phi(E) = E$, hence $\phi \circ \psi$ comes from an element~$g_{E}$ as required.

\begin{lemme}
\label{lmDescription des spheres en termes de doubles classes}
For all $[R] \in [\Delta_{F}]$, we have
\[ \Upsilon_{F}^{-1}(\mathcal{S}^{F}(F,[R])) = K_{F} g_{R} K_{F}/ K_{F} \quad \text{and} \quad \Upsilon_{F}^{-1}(\mathcal{S}^{F}_{\opp}(F,[R]))=K_{F} g_{R}^{-1} K_{F}/ K_{F}.
\]
\end{lemme}
\begin{proof}
For any $E \in \mathcal{S}^{F}(F,[R])$, there exists some $g \in K_{F}$ such that $g\cdot E = R = g_{R}\cdot F$, hence $\Upsilon_{F}^{-1}(E)$ belongs to $K_{F} g_{R} K_{F}/K_{F}$. Now let $x = k_{1}g_{R}k_{2}$ be an element of $K_{F}g_{R}K_{F}$: then we have $\Upsilon_{F}(x):= \Upsilon_{F}(k_{1}g_{R}) = k_{1}g_{R}\cdot F = k_{1}\cdot R$. As $d^{F}$ is $G$-invariant, we obtain that
\[d^{F}(k_{1}\cdot F, k_{1}\cdot R) = d^{F}(F,R) = d^{F}(F, k_{1}\cdot R),
\]
hence $x$ is in $\Upsilon^{-1}_{F}(S^{F}(F,[R]))$ and the proof of the first equality is complete. The proof of the second equality is similar and left to the reader.
\end{proof}

\subsection{Hecke algebra associated with a spherical type $0$ face}
Let $C$ and $C'$ be two positive type $0$ chambers base at some common vertex $x \in \I_{0}:= G\cdot0$. We can (and will) identify $W$ with the set of type $0$ chambers of $\A$ whose vertex lies in $Y^{+}$. Thus $d^{W}(C,C') = d^{C^{+}_{0}}(C,C')$ is in $W^{v}$ and we can set $d(C,C'):= \ell(d^{W}(C,C'))$.
\begin{lemme}
\label{lemFinitude des boules de chambres}
Let $C$ be a positive type $0$ chamber of $\I$ and let $x$ be its vertex. For any integer $n \geq 0$, the set $B_{n}(C)$ of all positive type $0$ chambers $C'$ of $\I$ based at $x$ and such that $d(C,C') \leq n$ is a finite set.
\end{lemme}
\begin{proof}
The argument goes by induction on $n \geq 0$, noticing that $B_{n}(C)$ contains $B_{m}(C)$ each time we have $n \geq m \geq 0$. As $\I$ is of finite thickness, the set $B_{1}(E)$ is finite for all $E \in G\cdot C^{+}_{0}$. Now let $n \geq 0$ be such that $B_{n}(E)$ is finite for all $E \in G\cdot C^{+}_{0}$ and take $C' \in B_{n+1}(C)$. By \cite[Prop.\,5.1]{rousseau2011masures}, we can choose an apartment $A$ that contains $C$ and $C'$. Let $\phi: (A,C) \to (\A,C^{+}_{0})$ be an isomorphism of apartments: then we have $\phi(C') = w\cdot C^{+}_{0}$ for some $w \in W^{v}$ of length at most $n+1$. We can assume that $\ell(w) = n+1$, otherwise $C'$ is in $B_{n}(C)$ and there is nothing more to do. In this case, let $\tilde{w} \in W^{v}$ be such that $\ell(\tilde{w}) = n$ and $d(\tilde{w}\cdot C_{0}^{+}, \phi(C'))=1$. Then $\tilde{C}:= \phi^{-1}(\tilde{w}\cdot C^{+}_{0})$ satisfies $d(C', \tilde{C}) = 1$, hence $C'$ belongs to $ \bigcup_{C'' \in B_{n}(C)}B_{1}(C'')$, which is a finite set, and the proof is complete.
\end{proof}

\subsubsection{Type of a type $0$ face}
Let $\F^{v}_{\A}$ be the set of all positive vectorial faces of $\A$ and $\F^{0}_{\A}$ be the set of all positive type $0$ faces of $\A$ based at $0$.
\begin{lemme}
\label{lemBijection entre les faces locales et vectorielles}
The map $f: \F^{v}_{\A} \to \F^{0}_{\A}$ that sends $F^{v} \in \F^{v}$ to $F^{\ell}(0,F^{v}) \in \F^{0}_{\A}$ is bijective.
\end{lemme}
\begin{proof}
The definition of local faces ensures that $f$ is well-defined and surjective. Now let $F^{v}_{1}, F^{v}_{2}$ be two distinct elements of $F^{v}_{\A}$. As $0$ is special, we have $F_{i}^{v} \in f(F_{i}^{v})$ for $i \in \{1,2\}$. But $F^{v}_{1} \cap F^{v}_{2} = \emptyset$ implies that $f(F^{v}_{1}) \not= f(F^{v}_{2})$ (for otherwise, we would have $\emptyset \in f(F^{v}_{1})$, which does not make sense) and $f$ is thus injective, which ends the proof.
\end{proof}

For any positive type $0$ face $F'$ of $\I$, there is some type $0$ face $F_{1} \leq C^{+}_{0}$ and some $g_{1} \in G$ such that $F'= g_{1}\cdot F_{1}$. The set $J \subset I$ such that $F_{1} = F^{\ell}(0, F^{v}(J))$ is called the \textbf{type of $F'$} and denoted by $\tau(F')$. This notion is well-defined: indeed, if we also have $F' = g_{2}F^{\ell}(0, F^{v}(J'))$ for some $g_{2} \in G$ and $J' \subset I$, then $g:= g_{2}^{-1}g_{1}$ is such that $F^{\ell}(0, F^{v}(J)) = g\cdot F^{\ell}(0, F^{v}(J')$. By (MA2) and \cite[2.2.1)]{rousseau2011masures}, we can assume that $g$ lies in $N$, hence $g_{|\A}$ is in $W^{v}$ and Lemma~\ref{lemBijection entre les faces locales et vectorielles} then implies that $F^{v}(J) = F^{v}(J')$. By \cite[1.3]{rousseau2011masures}, this requires that $J = J'$, as wanted.
\begin{rque}
\label{rqueInvariance type par G}
The type of a face is invariant under the action of $G$. Also note that for any type $0$ chamber $C$ and any subset $J$ of $I$, there exists exactly one sub-face of~$C$ with type $J$.
\end{rque}

\subsubsection{Finiteness results for spherical type $0$ faces}
\textbf{From now on, we assume that the face $F$ is furthermore spherical}.
\begin{lemme}
\label{lemFinitude du nombre de chambres contenant une face}
For all $F' \in \Delta_{F}$, the set $\mathcal{C}_{F'}$ of all type $0$ chambers of $\I$ containing~$F'$ is finite.
\end{lemme}
\begin{proof}
Fix a chamber $C \in \mathcal{C}_{F'}$, denote by $x$ its vertex and pick another chamber~$C'$ in $\mathcal{F'}$. By \cite[Prop.\,5.1]{rousseau2011masures}, there exists an apartment $A$ that contains $C$ and $C'$. We identify $A$ with $\A$ and fix the origin of $\A$ at $x$: then there exists $w \in W^{v}$ such that $C' = w\cdot C$. If $J$ denotes the type of $F'$, then $w\cdot F'$ is also a sub-face of type $J$ in $C'$, hence we have $w\cdot F' = F'$ by the unicity property stated in Remark \ref{rqueInvariance type par G}. This means that $w$ belongs to $W_{F'}$, which is a finite group as $F' \in \Delta_{F}$ is spherical (because~$F$~is). In particular, we must have $d(C,C') \leq \max\{\ell(u), \ u \in W_{F'}\}$, which ends the proof by Lemma~\ref{lemFinitude des boules de chambres}.
\end{proof}

\begin{lemme}
\label{lemCaracterisation du rayon}
Let $(E_{1},E_{2})$ and $(E'_{1}, E'_{2})$ be in $\Delta_{F}\times_{\leq}\Delta_{F}$. Then $d^{F}(E_{1}, E_{2}) = d^{F}(E'_{1}, E'_{2})$ iff there exists an isomorphism $\phi: (E_{1}, E_{2}) \to (E'_{1}, E'_{2})$.
\end{lemme}
\begin{proof}
Assume that $d^{F}(E_{1}, E_{2}) = d^{F}(E'_{1}, E'_{2}) =: [R]$. For any choice of $$\psi: (E_{1}, E_{2}) \to (F,R)\quad\text{and}\quad\psi': (E_{1}, E_{2}) \to (F,R),$$ the map $\phi:= \psi'^{-1} \circ \psi$ satisfies the required property.

Conversely, suppose that there exists an isomorphism $\phi:(E_{1},E_{2})\mapsto (E'_{1},E'_{2})$. Pick $R \in d^{F}(E_{1},E_{2})$ and choose $\psi: (E_{1}, E_{2}) \to (F,R)$: then $\phi^{-1} \circ \psi$ is an isomorphism that maps $(E'_{1}, E'_{2})$ to $(F,R)$. By Proposition \ref{propDefinition de la F-distance}, we thus have $d^{F}(E'_{1},E'_{2})=[R]=d^{F}(E_{1},E_{2})$, which ends the proof of the lemma.
\end{proof}

\begin{lemme}
\label{lemChambres contenant des faces a distance fixee}
Let $(F_{1},F_{2})\in \Delta_{F}\times_\leq \Delta_{F}$ and $r:=d^{F}(F_{1},F_{2})$. Let $R \in \Delta_{\geq F}^{\A}$ be such that $r = [R] = W_{F}\cdot R$ and let $\mathcal{C}_{\A}(r)$ be the set of all chambers of $\A$ containing an element of $r$. For any type $0$ chambers $C_{1}$ and $C_{2}$ respectively dominating $F_{1}$ and $F_{2}$, we have $d^{W}(C_{1}, C_{2}) \in \mathcal{C}_{\A}(r)$. Moreover, the set $\mathcal{C}_{\A}(r)$ is finite.
\end{lemme}
\begin{proof}
Pick an apartment $A$ containing $C_{1}$ and $C_{2}$ and an isomorphism $$\phi: (A, C_{1}) \to (\A, C_{0}^{+}).$$ By Remark \ref{rqueInvariance type par G}, $\phi(F_{1})$ is the unique sub-face of type $\tau(F)$ in $C^{+}_{0}$, hence we have $F = \phi(F_{1})$ and $\phi(F_{2})$ belongs to $W_{F}\cdot R = r$, which implies that $d^{W}(C_{1}, C_{2})$ lies in $\mathcal{C}_{\A}(r)$.

Now note that the map sending a positive type $0$ face $F'$ on its type $\tau(F')$ induces a bijection from the set of type $0$ chambers of $\A$ containing $w\cdot R$ onto the fixer $W_{R}$ of $R$ in $W$. As $W_{R}$ is a conjugate of $W_{F}$, it is finite (as $F$ is spherical), hence so is $\mathcal{C}_{\A}(r)$.
\end{proof}

\begin{lemme}
\label{lemFinitude des parametres de l'algebre}
For any pair $(F_{1}, F_{2}) \in \Delta_{F}\times_\leq \Delta_{F}$ and any elements $r_{1}, r_{2} \in [\Delta_{F}]$, the set $\mathcal{S}^{F}(F_{1},r_{1})\cap \mathcal{S}^{F}_{\opp}(F_{2},r_{2})$ is finite and its cardinality only depends on $r_{1}, r_{2}$ and $r:=d^{F}(F_{1}, F_{2})$.
\end{lemme}
\begin{proof}
Denote by $\mathcal{S}$ the set of type $0$ chambers containing an element of $\mathcal{S}^{F}(F_{1},r_{1})\cap \mathcal{S}^{F}_{\opp}(F_{2},r_{2})$ and let $C_{1}$ (resp. $C_{2}$) be a type $0$ chamber that contains $F_{1}$ (resp. $F_{2}$). By Lemma~\ref{lemChambres contenant des faces a distance fixee}, any chamber $C \in \mathcal{S}$ satisfies $d^{W}(C_{1},C) \in \mathcal{C}_{\A}(r_{1})$ and $d^{W}(C,C_{2}) \in \mathcal{C}_{\A}(r_{2})$. This implies that $\mathcal{S}$ is contained in
\[ \bigcup_{(w_{1}, w_{2}) \in \mathcal{C}_\A(r_{1}) \times \mathcal{C}_\A(r_{2})}\hspace*{-2mm} \{C\in \mathcal{C}_{0}^{+} \mid C_{1} \leq C \leq C_{2},\ d^{W}(C_{1},C)=w_{1} \mathrm{\ and\ }d^{W}(C,C_{2})=w_{2}\}.\]
By \cite[Prop.\,2.3]{bardy2016iwahori} and Lemma~\ref{lemChambres contenant des faces a distance fixee}, this inclusion implies the finiteness of $\mathcal{S}$, which itself implies that $\mathcal{S}^{F}(F_{1},r_{1})\cap \mathcal{S}^{F}_{\opp}(F_{2},r_{2})$ is finite.

To prove the independence of the cardinality, assume that $(F'_{1}, F'_{2}) \in \Delta_{F}\times_\leq \Delta_{F}$ is such that $d^{F}(F'_{1}, F'_{2}) = r$. By Lemma~\ref{lemCaracterisation du rayon}, there exists an isomorphism $$\phi: (F_{1}, F_{2}) \to (F'_{1}, F'_{2}).$$ Thus we have $$\mathcal{S}^{F}(F'_{1},r_{1}) \cap \mathcal{S}^{F}_{\opp}(F'_{2},r_{2}) = \phi\left(\mathcal{S}^{F}(F_{1},r_{1})\cap \mathcal{S}^{F}_{\opp}(F_{2},r_{2})\right),$$ which ends the proof.
\end{proof}
Following Lemma \ref{lemFinitude des parametres de l'algebre}, and with the same notations, we set $$a^{r}_{r_{1}, r_{2}}:= \vert \mathcal{S}^{F}(F_{1},r_{1})\cap \mathcal{S}^{F}_{\opp}(F_{2},r_{2})\vert.$$
\begin{lemme}
\label{lemFinitude des supports}
For any elements $r_{1}, r_{2}$ in $[\Delta_{F}]$, the set
\begin{multline*}
P_{r_{1}, r_{2}}:= \bigl\{d^{F}(F_{1},F_{2}), \ (F_{1},F',F_{2})\in \Delta_{F}\times_\leq \Delta_{F}\times_\leq \Delta_{F} \mid\\ d^{F}(F_{1},F')=r_{1} \mathrm{\ and\ }d^{F}(F',F_{2})=r_{2}\bigr\}
\end{multline*}
is finite.
\end{lemme}
\begin{proof}
Denote by $\mathcal{E}$ the set of all triples $(C_{1}, C', C_{2})$ of type $0$ chambers such that there exist sub-faces $F_{1} \subset C_{1}$, $F' \subset C'$ and $F_{2} \subset C_{2}$ of these chambers that satisfy $d^{F}(F_{1},F') = r_{1}$ and $d^{F}(F', F_{2}) = r_{2}$. If $(C_{1}, C', C_{2})$ is in $\mathcal{E}$, then Lemma~\ref{lemChambres contenant des faces a distance fixee} implies that $d^{W}(C_{1},C') \in \mathcal{C}_{\A}(r_{1})$ and $d^{W}(C',C_{2}) \in \mathcal{C}_{\A}(r_{2})$. This proves that $P:= \{d^{W}(C_{1}, C_{2}), \ (C_{1},C',C_{2})\in \mathcal{E}\}$ is contained in $ \bigcup_{(\mathbf{w_{1}}, \mathbf{w_{2}}) \in \mathcal{C}_{\A}(r_{1}) \times \mathcal{C}_{\A}(r_{2})} P_{\mathbf{w_{1}},\mathbf{w_{2}}}$, where the $P_{\mathbf{w_{1}},\mathbf{w_{2}}}$'s are the finite sets introduced in the proof of \cite[Prop.\,2.2]{bardy2016iwahori}. By Lemma \ref{lemChambres contenant des faces a distance fixee}, this proves that $P$ is finite.

Let $(F_{1}, F', F_{2})\!\in\!\Delta_{F}\times_\leq \Delta_{F} \leq \Delta_{F}$ be such that $d^{F}(F_{1},F')\!=\!r_{1}$ and $d^{F}(F', F_{2})\!=\!r_{2}$. Then there is a triple $(C_{1}, C', C_{2}) \in \mathcal{E}$ such that $F_{i}$ is a face of $C_{i}$ for $i \in \{1,2\}$. The distance $d^{F}(F_{1}, F_{2})$ is of the form $W_{F}\cdot F''$ for some face $F''$ of $d^{W}(C_{1}, C_{2})$, hence the lemma follows.
\end{proof}

\subsubsection{Definition of the Hecke algebra $\FHH$}
Let $R$ be a commutative unitary ring.\footnote{Note that we do not require here any of the additional assumptions made on $\RR$ in Section~\ref{secAlgebre d'IH completee}.} Denote by $\FHH = \FHH^{\I}_{R}$ the set of all functions $\varphi: G\backslash\Delta_{F}\times_{\leq} \Delta_{F} \to R$. For any $r \in [\Delta_{F}]$, let $T_{r}: \Delta_{F}\times_{\leq} \Delta_{F} \to R$ be defined as follows (where $\delta_{.,.}$ denotes the Kronecker symbol):
\[ \forall (F_{1}, F_{2}) \in \Delta_{F}\times_{\leq} \Delta_{F}, \ T_{r}(F_{1}, F_{2}):= \delta_{d^{F}(F_{1},F_{2}), r}.\]
One directly checks that $\FHH$ is a free $R$-module with basis $\{T_{r}, \ r \in [\Delta_{F}]\}$.
\begin{thm}
\label{thmDefinition de la convolution parahorique}
Define a product $*: {\FHH} \times {\FHH} \to {\FHH}$ by the following formula:
\[ \forall (\varphi_{1}, \varphi_{2}) \in {\FHH} \times {\FHH}, \quad \varphi_{1} * \varphi_{2}:= \biggl[(F_{1}, F_{2}) \mapsto\hspace*{-3mm} \sum_{F' \in \Delta_{F} \mid F_{1} \leq F' \leq F_{2}}\hspace*{-3mm}\varphi_{1}(F_{1},F')\varphi_{2}(F',F_{2})\biggr].\]
Then the product $*$ is well-defined and endow $\FHH$ with a structure of associative algebra that has $T_{[F]}$ for identity element. Moreover, the product of any two elements of the basis $\{T_{r}, \ r \in [\Delta_{F}]\}$ is given by the following formula:
\[ \forall (r_{1}, r_{2}) \in [\Delta_{F}] \times [\Delta_{F}], \quad T_{r_{1}} * T_{r_{2}} = \sum_{r \in P_{r_{1}, r_{2}}} a^{r}_{r_{1},r_{2}} T_{r}.\]
\end{thm}
\begin{proof}
Lemmas~\ref{lemFinitude des parametres de l'algebre} and \ref{lemFinitude des supports} imply that $*$ is well-defined and give the required formula for $T_{r_{1}}*T_{r_{2}}$ for any $(r_{1}, r_{2}) \in [\Delta_{F}] \times [\Delta_{F}]$. The associativity of $*$ directly comes from the definition, and a direct computation shows that $T_{[F]}$ is the identity element as we have $\mathcal{S}^{F}(F_{1},[F])=\{F_{1}\}$ for all $F_{1}\in \Delta_F$.
\end{proof}

\begin{defn}
\label{DefAlgebre de Hecke pour face spherique type $0$}
The algebra $\FHH={\FHH^{\I}_{R}}$ is called the \textbf{Hecke algebra of $\I$ associated to $F$ (or: to $K_{F}$) over $R$}.
\end{defn}

\begin{rque}
Given $g \in G^{+}$, there exists some element $R_{g} \in \Delta_{\geq F}^{\A}$ such that
\[ \{F'\in K_{F}gK_{F}\cdot F, \ F'\subset \A\} = [R_{g}].
\]
Let $(F_{1}, F_{2}) \in G\backslash \Delta_{F}\times_{\leq} \Delta_{F}$. We can always assume that $F_{1} = F$ and write $F_{2} = g\cdot F$ for some $g \in G$. One easily checks that $f(g):= K_{F}gK_{F}$ only depends on $(F_{1}, F_{2})$, and that the corresponding map $f: G \backslash \Delta_{F} \times_{\leq} \Delta_{F} \to K_{F} \backslash G^{+}/K_{F}$ is bijective. Via~$f$, we can identify $\FHH$ with the set of all functions $K_{F}\backslash G^{+}/K_{F} \to R$. Under this identification, $T_{R_{g}}$ corresponds to $e_{g}:= \mathds{1}_{K_{F}gK_{F}}$ for all $g \in G^{+}$. Moreover, for any $g,g' \in K_{F} \backslash G^{+} /K_{F}$, we have
\[ e_{g}*e_{g'} = \sum_{g'' \in K_{F} \backslash G^{+} /K_{F}} m(g,g'; g'')e_{g''},\]
where $m(g,g';g'') = a^{[R_{g''}]}_{[R_{g}], [R_{g'}]}$ for all $g'' \in K_{F} \backslash G^{+} /K_{F}$. Using Lemmas~\ref{lmDescription des spheres en termes de doubles classes} and \ref{lemFinitude des parametres de l'algebre}, we get that $m(g,g';g'') = |(KgK\cap g''Kg'^{-1} K)/K|$, as in the reductive case (compare with \eqref{FormuleProduitConvolutionAlgebreHecke}).
\end{rque}

\begin{rque}
For now, we do not know whether it is possible to define a completed Hecke algebra $\widehat{\FHH}$ for any spherical face $F$ as above in the similar manner as what we did for the Iwahori-Hecke algebra. To generalize our completion process to this context, one would in particular need an analogue of Bernstein-Lusztig relations for arbitrary $F$.
\end{rque}

\subsection{What about non-spherical type $0$ faces?}
\label{subsecCasnonspherique}
In \cite{gaussent2014spherical}, Gaussent and Rousseau defined the spherical Hecke algebra as a Hecke algebra associated with the non-spherical type $0$ face $F_{0}$, and we noticed in Remark~\ref{rqueComparaison aux autres distances vectorielles} that their distance $d^{v}$ matches with our $d^{F_{0}}$. Consequently, it seems natural to try to associate a Hecke algebra with any type $0$ face $F$ between $F_{0}$ and $C^{+}_{0}$, \textit{i.e} to see whether the extra assumption of being spherical can be suppressed.

In this section, we consider a non-spherical type $0$ face $F$ such that $F_{0}\subsetneq F\subsetneq C_{0}^{+}$. Note that this implies that $A$ is an indefinite Kac-Moody matrix of size $\geq 3$: indeed, when $A$ is of finite type, then any type $0$ face is spherical, and when $A$ is of affine type, the only non-spherical type $0$ face of $C^{+}_{0}$ is $F_{0}$. In this last section, we will prove that the coefficients involved in the definition of the convolution product introduced earlier (see Theorem \ref{thmDefinition de la convolution parahorique}) are now infinite. The proof of this result requires the injectivity of the restriction map that sends $w \in W^{v}$ to $w_{|Q^{\vee}}$. This property is proved in \cite{kac1994infinite} for less general realizations of $A$ than the one we use, hence we will start by extending this property to our framework: this is the point of Lemma \ref{lemRestriction du groupe de Weyl aux copoids} below. 

\subsubsection{Realizations of a Kac-Moody matrix}
Let $A = (a_{i,j})_{i,j\in \llbracket 1,n\rrbracket}$ be a Kac-Moody matrix. Following \cite[Chap.\,1]{kac1994infinite}, we say that a \textit{realization of $A$} is a triple $(\AAA, \Pi, \Pi^{\vee})$ where $\AAA$ denotes an $\R$-vector space,\footnote{Note that in \cite{kac1994infinite}, complex vector spaces are used instead of real vector spaces.} $\Pi = \{ \alpha_{1}, \ldots, \alpha_{n}\}$ a family of $n$ elements in~$\AAA^{*}$ (the dual space of $\AAA$) and $\Pi^{\vee} = \{\alpha_{1}^{\vee}, \ldots, \alpha_{n}^{\vee}\}$ a family of $n$ elements in $\AAA$, such that the following three properties hold.
\begin{enumerate}
\item[(F)]
The elements of $\Pi$ (resp. $\Pi^{\vee}$) are linearly independent in $\AAA^{*}$ (resp. in $\AAA$).
\item[(C)]
For all $i,j \in \llbracket 1,n\rrbracket$, $\alpha_{j}(\alpha_{i}^{\vee}) = a_{i,j}$.
\item[(D)]
We have $n - \mathrm{rk}(A) = \dim_{\R}\AAA - n$.
\end{enumerate}
A \textbf{generalized free realization of $A$} is a triple $(\AAA, \Pi, \Pi^{\vee})$ with $\AAA, \Pi, \Pi^{\vee}$ defined as above but only satisfying properties (F) and (C). Two realizations $(\AAA_{1}, \Pi_{1}, \Pi^{\vee}_{1})$ and $(\AAA_{2}, \Pi_{2}, \Pi_{2}^{\vee})$ are said \textbf{isomorphic} if there exists an isomorphism of vector spaces $\phi: \AAA_{1} \to \AAA_{2}$ such that $\phi^{*}(\Pi_{1}) = \Pi_{2}$ and $\phi(\Pi^{\vee}_{1}) = \Pi_{2}^{\vee}$. We know by \cite[Prop.\,1.1]{kac1994infinite} that up to (non unique in general) isomorphism, $A$ admits a unique realization $(\AAA_{0}, \Pi_{0}, \Pi_{0}^{\vee})$.

Given a generalized free realization $(\AAA, \Pi, \Pi^{\vee})$ of $A$, we let the \textbf{inessential part of~$\AAA$} be the subspace $\AAA_\mathrm{in}:= \bigcap_{i = 1}^{n} \ker \alpha_{i}$. We also set $$Q^{\vee}_{\AAA}:= \bigoplus_{i = 1}^{n} \Z\alpha_{i}^{\vee}\quad\text{and}\quad Q^{\vee}_{\R, \AAA}:= \bigoplus_{i = 1}^{n} \R\alpha_{i}^{\vee}.$$

The next lemma is easy to prove and thus left to the reader.

\begin{lemme}
\label{lemDescription des realisations generales}
For any generalized free realization $\AAA$ of $A$, there exist subspaces $\AAA'\subset \AAA$ and $B\subset \AAA_\mathrm{in}$ such that $\AAA'$ is isomorphic to $\AAA_0$ (as realizations of $A$), $Q^\vee_\AAA \subset \AAA'$ and $\AAA=\AAA'\oplus B$.
\end{lemme}

Let $W^v_{\mathcal{A}}$ be the Weyl group of $\mathcal{A}$, \textit{i.e} the subgroup of $\mathrm{GL}(\mathcal{A})$ generated by the~$r_i$, $i\in I$ (where $r_i:\mathcal{A}\to \mathcal{A}$ sends any $x\in \mathcal{A}$ to $x-\alpha_i(x)\alpha_i^\vee$).

\begin{lemme}
\label{lemRestriction du groupe de Weyl aux copoids}
For any generalized free realization $\AAA$ of $A$, the map $$\left[ w \in W^{v}_{\AAA} \mapsto w_{|Q^{\vee}} \in \mathrm{Aut}_\Z(Q^{\vee}_{\AAA})\right]$$ is injective.
\end{lemme}

\begin{proof}
Write $\AAA=\AAA'\oplus B$ with $\AAA'$ and $B$ as in Lemma~\ref{lemDescription des realisations generales}. For any $x\in \AAA$ and $w\in W^{v}_\AAA$, we have $w(x)-x\in Q^{\vee}_{\R,\AAA}$, hence $\AAA'$ is stable under the action of $W^{v}_{\A}$. Moreover, $W^{v}_{\AAA}$ fixes pointwise $\AAA_\mathrm{in}$, hence the restriction map $W^{v}_\AAA\to W^{v}_{\AAA'}$ is a an isomorphism. As a consequence, we can assume that $\AAA = \AAA_0$. Now apply assertion (3.12.1) of the proof of \cite[Prop.\,3.12]{kac1994infinite} to $\Delta^{\vee}$ instead of $\Delta$: we get that the only $w \in W^{v}_{\AAA_{0}}$ satisfying $w_{|\Delta^{\vee}} = 1$ is $w = 1$. As $\Delta^{\vee}$ is contained in $Q^{\vee}_{\AAA}$, this ends the proof.
\end{proof}

\subsubsection{Infinite intersections of spheres}
From now on, we assume that $F$ is a non-spherical type $0$ face of $\A$ that satisfies $F_{0} \subsetneq F\subsetneq C^{+}_{0}$. Recall that this implies that the fixer $W_{F}$ of $W$ is infinite. Indeed, we can assume that $F$ has $0$ for vertex, which identifies $W_{F}$ with a subgroup of $W^{v}$. Let $F^{v}$ be the vectorial face such that $F = F^{\ell}(0, F^{v})$ and let us prove that $W_{F}$ is also the fixer $W_{F^{v}}$ of $F^{v}$ (which will prove the claim as $F^{v}$ is non spherical, hence $W^{F^{v}}$ is infinite by definition). If~$w \in W_{F}$, let $X \in F$ be fixed by $w$: then $w$ fixes $\R^{*}_{+}.X \supset F^{v}$ hence $w$ is in $W_{F^{v}}$. Conversely, we have $W_{F^{v}} \subset W_{F}$ because $F^{v} \in F$ (as $0$ is special), hence $W_{F} = W_{F^{v}}$ is infinite.
\begin{rque}
\label{rqueCNS pour avoir une orbite infinie}
By \cite[\S 1.3]{rousseau2011masures}, the vectorial faces based at $0$ form a partition of the Tits cone. Therefore, for any vectorial face $F^{v}$, if there exist some $u\in F^v$ and some $w\in W^v$ such that $w\cdot u\in F^{v}$, then $w\cdot F^v=F^v$. Consequently, for any $W'\subset W^{v}$, $W'\cdot F^{v}$ is infinite if and only if $W'\cdot u$ is infinite for some $u\in F^{v}$, if and only if $W'\cdot u$ is infinite for all $u\in F^{v}$.
\end{rque}
The proof of the next proposition uses the \textbf{graph of the matrix $A$}, whose vertices are the elements $i\in I$ and whose arrows are the pairs $\{i,j\}$ such that $a_{i,j}\neq 0$.

\begin{lemme}\label{lemNon definition du produit pour les non-spheriques}
Suppose that the matrix $A$ is indecomposable. For any non-spherical type $0$ face $F$ of $\A$ that satisfies $F_{0} \subsetneq F\subsetneq C^{+}_{0}$, there exists $w\in W^{v}$ such that $W_{F}\cdot w\cdot F$ is infinite.
\end{lemme}
\begin{proof}
Write $F = F^{\ell}(0,F^{v})$ with $$F^{v} = F^v(J)=\{ x \in \A \mid \forall j \in J, \, \alpha_{j}(x) =0 \text{ and } \forall i \not\in J, \alpha_{i}(x) > 0\}$$ for some subset $J$ of $I$. Note that $J\neq I$ as $F_{0}$ is strictly contained in $F$. Let $k \in I$ be such that $W_{F}\cdot\alpha_{k}^{\vee}$ is infinite (such a $k$ exists by Lemma~\ref{lemRestriction du groupe de Weyl aux copoids}). As the graph of $A$ is connected \cite[4.7]{kac1994infinite}, any element $j \in I\setminus J$ can be linked to $k$ via a finite sequence $j = j_{1}, \ldots, j_{\ell} = k$ of elements of $I$ that satisfy $ \prod_{m = 1}^{\ell-1}a_{j_{m}, j_{m+1}} \not= 0$. We fix such a $j$ and such a sequence $j_{1}, \ldots, j_{\ell}$.

Now pick $u \in F^{v}$ and let us show the existence of some $w \in W^{v}$ such that \hbox{$\alpha_{k}(w\cdot u) \neq 0$}. Given $x \in \A$ and $m \in \llbracket 1, \ell \rrbracket$, we say that $x$ satisfies $P_{m}$ when $\alpha_{j_{m}}(x) \not= 0$ and $\alpha_{j_{m'}}(x) = 0$ for all $m'\in \llbracket m+1,l\rrbracket$. If $x \in \A$ satisfies $P_{m}$ for some $m \in \llbracket 1,\ell - 1\rrbracket$, then $x':= r_{j_{m}}(x)$ satisfies $\alpha_{j_{m+1}}(x') = - \alpha_{j_{m}}(x)a_{j_{m}, j_{m+1}} \not= 0$ (recall that $x' = x - \alpha_{j_{m}}(x)\alpha_{j_{m}}^{\vee}$), hence $x'$ satisfies $P_{s}$ for some $s \in \llbracket m+1, \ell \rrbracket$. As $u$ is in $F^{v}$ and $j_{1} = j$ is in $I\setminus J$, we have $\alpha_{j_{1}}(u) >0$, hence $u$ satisfies $P_{m}$ for some $m \in \llbracket 1, \ell\rrbracket$. Replacing $u$ by $x$ in the previous argument and using successive iterations, we finally get some $w \in W^{v}$ such that $w\cdot u$ satisfies $P_{\ell}$, \textit{i.e} such that $\alpha_{k}(w\cdot u) \neq 0$.

We conclude as follows: if $W_{F}\cdot w\cdot u$ is finite, then $W_{F}.r_{k}(w\cdot u) = W_{F}.(u - \alpha_{k}(w\cdot u)\alpha_{k}^{\vee})$ is infinite, hence at least one of the orbits $W_{F}\cdot w\cdot u$ or $W_{F}.r_{k}(w\cdot u)$ is infinite, which implies the required result by Remark~\ref{rqueCNS pour avoir une orbite infinie}.
\end{proof}
Let $A_{1}, \ldots, A_{r}$ be the indecomposable components of the Kac-Moody matrix $A$. For any $i \in \llbracket 1, r \rrbracket$, pick a realization $\AAA_{i}$ of $A_{i}$: then $\A = \AAA_{1} \oplus \ldots \oplus \AAA_{r}$. Also note that $W^{v}$ decomposes as $W^{v} = W^{v}_{1} \times \ldots \times W^{v}_{r}$, where $W^{v}_{i}$ denotes the vectorial Weyl group of $\AAA_{i}$, and that we can decompose any face $F'$ of $\A$ as $F'= F'_{1} \oplus \ldots \oplus F'_{r}$ with $F'_{i} \subset \AAA_{i}$.

\begin{prop}\label{propNon definition du produit pour les non-spheriques}
Let $F' = \bigoplus_{i = 1}^{r} F'_{i}$ be a type $0$ face of $\A$. The following are equivalent.
\begin{enumerate}\renewcommand{\theenumi}{\roman{enumi}}
\item\label{enum:i}
There exists $w \in W^{v}$ such that $W_{F'}\cdot w\cdot F'$ is infinite.
\item\label{enum:ii}
There exists $i \in \llbracket 1,r \rrbracket$ such that $F'_{i}$ is non-spherical and different from $F_{i,0}:= F^{\ell}(0,\AAA_{i, in})$. (Recall that $F_{i,0}$ is the minimal type $0$ face of $\AAA_{i}$ based at $0$.)
\end{enumerate}
\end{prop}

\begin{proof}
The decomposition of $F'$ induces a decomposition of its fixer as $W_{F'} = W_{F'_{1}} \times \ldots \times W_{F'_{r}}$. First assume the existence of some $w \in W^{v}$ such that $W_{F'}\cdot w\cdot F'$ is infinite and decompose $w$ as $w = (w_{1}, \ldots, w_{r})$. Then $$W_{F'}\cdot w\cdot F' = W_{F'_{1}}\cdot w_{1}\cdot F'_{1} \oplus \ldots \oplus W_{F'_{r}}\cdot w_{r}\cdot F'_{r},$$ hence there is (at least) an integer $i \in \llbracket 1, r \rrbracket$ such that $W_{F'_{i}}\cdot w_{i}\cdot F'_{i}$ is infinite. For such an $i$, $F'_{i}$ must be non-spherical (otherwise $W_{F'_{i}}$ would be finite) and different from~$F_{i,0}$ (otherwise $W_{F'_{i}}\cdot w_{i}\cdot F'_{i} = F_{i,0}$). Hence \eqref{enum:i} implies \eqref{enum:ii}. The reverse implication is a consequence of Lemma~\ref{lemNon definition du produit pour les non-spheriques}.
\end{proof}

The next proposition gives a counterexample to Lemma~\ref{lemFinitude des parametres de l'algebre} for non-spherical faces, which explains why we needed this restriction in our construction.
\begin{prop}
\label{propIntersection spheres infinies}
Let $F'$ be a face based at $0$ for which there exists some $w \in W^{v}$ such that $W_{F'}w\cdot F'$ is infinite. Then $\mathcal{S}^{F'}(F',[w\cdot F'])\cap \mathcal{S}_{\opp}^{F'} (F',[w^{-1}\cdot F'])$ is infinite.
\end{prop}
\begin{proof}
It is enough to check that $W_{F'}w\cdot F'$ is contained in $$\mathcal{S}^{F'}(F',[w\cdot F'])\cap \mathcal{S}_{\opp}^{F'} (F',[w^{-1}\cdot F']).$$ Let $E \in W_{F'}\cdot w\cdot F'$ and let $w_{E} \in W_{F'}$ be such that $E= w_{E}\cdot w\cdot F'$. As $w_{E}\cdot F' = F'$, we have $F' \leq E \leq F'$, hence $d^{F'}(F',E) = [E] = [w\cdot F']$ by definition of $d^{F'}$. Now the isomorphism $(w_{E}\cdot w)^{-1}: \A \to \A$ maps $E$ to $F'$ and $F'$ to $w^{-1}\cdot w_{E}^{-1}\cdot F'=w^{-1}\cdot F'$, thus $d^{F'}(E,F') = [w^{-1}\cdot F']$. This shows that $E$ belongs to $$\mathcal{S}^{F'}(F',[w\cdot F'])\cap \mathcal{S}_{\opp}^{F'} (F',[w^{-1}\cdot F']),$$ hence the proposition.
\end{proof}

Recall that the notations $Y^{f}$ and $Y^{\infty}_\mathrm{in}$ used in the next result were introduced in Section \ref{subsubsectionCalcul des centres}.
\begin{cor}
\label{corOrbites infinies}
Let $\lambda \in Y^{+}$. Its $W^{v}$-orbit $W^{v}\cdot\lambda$ is finite iff $\lambda$ belongs to $Y^{f} \oplus Y^{\infty}_\mathrm{in}$.
\end{cor}
\begin{proof}
Given $\lambda \in Y^{+}$, write $\lambda = \sum_{j = 1}^{r} \lambda_{j}$ with $\lambda_{j} \in \AAA_{j}$ for all $j\in \llbracket 1,r\rrbracket$. First assume that $\lambda$ is in $Y^{f} \oplus Y^{\infty}_\mathrm{in}$: then $$W^{v}\cdot\lambda = \bigoplus_{j \in J^{f}} W^{v}_{j}\cdot\lambda_{j} \oplus \bigoplus_{j \in J^{\infty}}\lambda_{j}.$$ As $W^{v}_{j}$ is finite for any $j \in J^{f}$, the finiteness of $W^{v}\cdot\lambda$ follows from its decomposition above and the converse implication is proved.

Now assume that $\lambda$ is not in $Y^{f} \oplus Y^{\infty}_\mathrm{in}$. Let $j \in J^{\infty}$ be such that $\lambda_{j} \not\in \AAA_{j, \mathrm{in}}$ and let $F^{v}_{j}$ be the vectorial face of $\AAA_{j}$ that contains $\lambda_{j}$. By Remark~\ref{rqueCNS pour avoir une orbite infinie}, the map $W^{v}_{j}\cdot F^{v}_{j} \lambda_{j} \to W^{v}_{j}\lambda_{j}$ that sends $w\cdot F^{v}_{j}$ onto $w\cdot \lambda_{j}$ is well-defined and bijective. If $F^{v}_{j}$ is spherical, then its stabilizer is finite and $W^{v}_{j}\cdot F^{v}_{j}$ is thus infinite as $W^{v}_{j}$ is. If $F^{v}_{j}$ is non-spherical, then Lemma~\ref{lemNon definition du produit pour les non-spheriques} produces an element $w_{j} \in W^{v}_{j}$ such that $W_{F_{j}}\cdot w_{j}\cdot F^{v}_{j}$ is infinite, where $W_{F_{j}}$ is also the fixer of $F^{v}_{j}$ in $W^{v}_{j}$. In any case, $W^{v}_{j}\cdot F^{v}_{j}$ is infinite, hence so is $W^{v}_{j}\cdot\lambda_{j}$, which ends the proof.
\end{proof}

\bibliographystyle{jepalpha}

\bibliography{abdellatif-hebert}
\end{document}